\documentclass{siamltex}
\usepackage[latin1]{inputenc}
\usepackage[T1]{fontenc}

\usepackage[centertags]{amsmath}

\usepackage{amssymb}
\usepackage{amsfonts,a4}
\usepackage{amsxtra}
\usepackage{amscd}
\usepackage{mathabx} 
\usepackage{esint}

\usepackage[english]{babel}

\usepackage{algorithm}
\usepackage{algorithmic}

\usepackage{color}
\usepackage[mathscr]{eucal}
\DeclareMathAlphabet{\mathpzc}{OT1}{pzc}{m}{it}

\usepackage{bbm}
\usepackage{enumitem}

\usepackage{ifpdf}
\ifpdf
  \usepackage[bookmarks=true]{hyperref}
  \usepackage[pdftex]{graphicx}
  \DeclareGraphicsExtensions{.pdf,.png,.jpg}
  \graphicspath{{./pic/}}

\else
  \usepackage[xdvi,dvips]{graphicx}
  \DeclareGraphicsExtensions{.ps,.eps}
\fi

\usepackage{graphicx}
\graphicspath{{./plots/}}

\usepackage{xspace} %

\newtheorem{thm}{Theorem}
\newtheorem{cor}[thm]{Corollary}
\newtheorem{lem}[thm]{Lemma}
\newtheorem{prop}[thm]{Proposition}

\newtheorem{rem}[thm]{Remark}
\newtheorem{ass}[thm]{Assumption}
\newtheorem{ex}[thm]{Example}

\numberwithin{equation}{section}

\newcommand{\Ax}{\mathcal{A}}

\newcommand{\Ao}{\ensuremath{\mathcal{A}}}
\newcommand{\Aon}[1][n]{\ensuremath{\Ao_{n}}}

\newcommand{\abs}[1]{\ensuremath{\left|#1\right|}}
\newcommand{\asdefined}{\mathrel{=:}}

\newcommand{\bsS}{\ensuremath{\boldsymbol{S}}}
\newcommand{\bSn}{\ensuremath{\boldsymbol{S}^n}}
\newcommand{\bsD}{\ensuremath{\boldsymbol{D}}}

\newcommand{\bss}{\ensuremath{\boldsymbol{s}}}
\newcommand{\bsd}{\ensuremath{\boldsymbol{d}}}
\newcommand{\bbI}{\ensuremath{\boldsymbol{1}}}

\newcommand{\Bogn}[1][n]{\mathfrak{B}^{#1}}

\DeclareMathOperator{\curl}{curl}

\renewcommand{\d}{\ensuremath{{\rm d}}}
\newcommand{\dt}{\ensuremath{\d t}}
\newcommand{\dx}{\ensuremath{\d x}}
\newcommand{\dy}{\ensuremath{\d y}}
\newcommand{\definedas}{\mathrel{:=}}

\newcommand{\dual}[2]{\ensuremath{\left\langle #1,\,#2\right\rangle}}

\DeclareMathOperator{\diam}{diam}
\DeclareMathOperator{\divo}{div}

\DeclareMathAlphabet{\lf}{OT1}{pzc}{m}{it}

\newcommand{\eg}{e.\,g.\xspace}
\newcommand{\elm}{\ensuremath{E}\xspace}

\newcommand{\grid}{\mathcal{G}}
\newcommand{\gridn}[1][n]{\grid_{#1}}

\newcommand{\Gx}{\ensuremath{\boldsymbol{G}_x}}

\newcommand{\helm}[1][\ell]{\ensuremath{h_{\elm}}}

\newcommand{\ie}{\hbox{i.\,e.},\xspace}

\newcommand{\mat}[1]{\boldsymbol{#1}}
\newcommand{\mcD}{\mathcal{D}}
\newcommand{\mcU}{\mathcal{U}}
\newcommand{\mcH}{\mathcal{H}}
\newcommand{\mcC}{\mathcal{C}}

\newcommand{\mcQ}{\mathcal{Q}}

\newcommand{\Mspace}[1][\Rdds]{\ensuremath{\mathfrak{M}({#1})}}

\newcommand{\new}[1]{{\color{black}{#1}}}
\newcommand{\nablas}{\boldsymbol{D}}
\newcommand{\N}{\ensuremath{\mathbb{N}}}

\newcommand{\nodes}{\ensuremath{\mathcal{N}}}

\newcommand{\norm}[2][\Omega]{\ensuremath{\|#2\|_{#1}}}

\providecommand{\Poincare}{{Poincar{\'e}}\xspace}
\renewcommand{\P}{\ensuremath{\mathbb{P}}}

\renewcommand{\paragraph}[1]{\noindent\raisebox{0pt}[10pt][0pt]{\textbf{#1.}}}
\newcommand{\Pndiv}[1][n]{\ensuremath{\Pi^{#1}_{\divo}}}
\newcommand{\PnQ}[1][n]{\ensuremath{\Pi^{#1}_\Q}}

\newcommand{\Q}{\ensuremath{\mathbb{Q}}}
\newcommand{\Qn}[1][n]{\ensuremath{\mathbb{Q}^{#1}}}

\newcommand{\R}{\ensuremath{\mathbb{R}}}
\newcommand{\Rd}{\ensuremath{\mathbb{R}^d}}\newcommand{\Rdd}{\ensuremath{\mathbb{R}^{d\times d}}}
\newcommand{\Rdds}{{\ensuremath{\Rdd_{\text{sym}}}}}

\newcommand{\scp}[3][\Omega]{\ensuremath{\left\langle #2,\,#3\right\rangle}_{#1}}

\newcommand{\side}{\ensuremath{S}\xspace}
\newcommand{\sides}{\mathcal{S}}

\newcommand{\sectionref}[1]{\S\ref{#1}}

\def\vec#1{\boldsymbol{#1}}

\newcommand{\tr}{\ensuremath{\tilde r}}

\newcommand{\Tri}{\mathcal{B}}
\newcommand{\Trilin}[4][]{\ensuremath{\Tri_{#1}[#2,\,#3,\,#4]}}

\newcommand{\Un}[1][n]{\ensuremath{\vec{U}^{#1}}}

\renewcommand{\vec}[1]{\ensuremath{\boldsymbol{#1}}}
\newcommand{\vecu}{\ensuremath{\vec{u}}}
\newcommand{\vece}{\ensuremath{\vec{e}}}
\newcommand{\vecE}{\ensuremath{\vec{E}}}
\newcommand{\vecv}{\vec{v}}
\newcommand{\vecV}{\vec{V}}
\newcommand{\vecVn}[1][n]{\vec{V}^{#1}}
\newcommand{\vecf}{\vec{f}}
\newcommand{\vecw}{\vec{w}}

\newcommand{\V}{\ensuremath{\mathbb{V}}}
\newcommand{\Vn}[1][n]{\ensuremath{\mathbb{V}^{#1}}}
\newcommand{\Vndiv}[1][n]{\ensuremath{\mathbb{V}^{#1}_{\divo}}}

\providecommand{\vec}[1]{\boldsymbol{#1}}

\newcommand{\weak}{\ensuremath{\rightharpoonup}}

\title{Finite element approximation
  of steady flows of incompressible fluids with\\ implicit power-law-like rheology}

\author{Lars Diening\thanks{Mathematisches Institut der Universit{\"a}t M{\"u}nchen,
 Theresienstrasse 39, D-80333 M{\"u}nchen, Germany ({\tt lars.diening@mathematik.uni-muenchen.de}).}
 \and Christian Kreuzer\thanks{Fakult{\"a}t f{\"u}r Mathematik, Ruhr-Universit{\"a}t Bochum,
 Universit{\"a}tsstrasse 150, D-44801 Bochum, Germany ({\tt christan.kreuzer@rub.de}).}
 \and Endre S{\"u}li\thanks{Mathematical Institute, University of Oxford,
  24-29 St Giles', Oxford OX1 3LB, UK ({\tt endre.suli@maths.ox.ac.uk}).}}

\begin{document}
\maketitle

\begin{abstract}
  We develop the analysis of finite element approximations of implicit
  power-law-like models for viscous incompressible fluids.  The Cauchy
  stress and the symmetric part of the velocity gradient in the class
  of models under consideration are related by a, possibly multi--valued, maximal monotone
  $r$-graph, with $1<r<\infty$. Using a variety of weak compactness
  techniques, including Chacon's biting lemma and Young measures, we
  show that a subsequence of the sequence of finite element solutions
  converges to a weak solution of the problem as the finite element
  discretization parameter $h$ tends to $0$. A key new technical tool
  in our analysis is a finite element counterpart of the Acerbi--Fusco
  Lipschitz truncation of Sobolev functions.
\end{abstract}



\pagestyle{myheadings}
\thispagestyle{plain}
\markboth{LARS DIENING, CHRISTIAN KREUZER, AND ENDRE S{\"U}LI}{FEM FOR IMPLICITLY CONSTITUTED FLUID FLOW MODELS}

\section{Introduction}
Most physical models describing fluid flow rely on the assumption that
the Cauchy stress is an explicit function of the symmetric part of the
velocity gradient of the fluid. This assumption leads to the
Navier--Stokes equations and its nonlinear generalizations, such as
various electrorheological flow models; see, e.\,g.,
\cite{Ladyzhenskaya:69,Lions:69,Ruzicka:00}.  It is known however that
the framework of classical continuum mechanics, built upon the notions
of current and reference configuration and an explicit constitutive
equation for the Cauchy stress, is too narrow to enable one to model
inelastic behavior of solid-like materials or viscoelastic properties
of materials. Our starting point in this paper is therefore a
generalization of the classical framework of continuum mechanics,
called the implicit constitutive theory, which was proposed recently
in a series of papers by Rajagopal; see, for example,
\cite{Rajagopal:03,Rajagopal:06}.  The underlying principle of the
implicit constitutive theory in the context of viscous flows is the
following: instead of demanding that the Cauchy stress is an explicit
function of the symmetric part of the velocity gradient, one may allow
an implicit and not necessarily continuous relationship between
these quantities. The resulting general theory therefore admits fluid
flow models with implicit and possibly discontinuous power-law-like
rheology; see, \cite{Malek:08,Malek:07}. Very recently a rigorous
mathematical existence theory was developed for these models by
Bul{\'\i}\v{c}ek, Gwiazda, M{\'a}lek, and {\'S}wierczewska-Gwiazda in
\cite{BulGwiMalSwi:09}. Motivated by the ideas in
\cite{BulGwiMalSwi:09}, we consider the construction of finite element
approximations of implicit constitutive models for incompressible
fluids and we develop the convergence theory of these numerical
methods by exploiting a range of weak compactness arguments.

Let $\Omega\subset \R^d$, $d\in\N$, be a bounded open Lipschitz domain with polyhedral boundary.
For $r\in(1,\infty)$ we define $r'\in(1,\infty)$ by the equality $\tfrac1r+\tfrac1{r'}=1$ and we set
\begin{align}\label{eq:tr}
  \tr \definedas \min\{r',r^*/2\}, \qquad\text{where}\qquad
  r^*\definedas
  \begin{cases}
    \tfrac{dr}{d-r}\quad&\text{if}~r<d,
    \\
    \infty\quad &\text{otherwise}.
  \end{cases}
\end{align}
We refer the reader to the first paragraph in Section \ref{s:inf-sup}
for the definitions of the various function spaces used in the paper
and for a list of our notational conventions.

\medskip

\paragraph{Problem}
For $\vecf\in W^{-1,r'}(\Omega)^d$ find
$(\vec{u},p,\mat{S})\in W_0^{1,r}(\Omega) ^d\times L_0^{
  \tr}(\Omega)\times L^{r'}(\Omega)^{d\times d}$ such that
\begin{align}\label{eq:implicit}
  \begin{alignedat}{2}
    \divo(\vec{u}\otimes\vec{u} +  p\bbI -\vec{S})&=\vecf &\quad
    &\text{in} \quad \mcD'(\Omega),
    \\
    \divo\vec{u} &=0& \quad&
    \text{in} \quad \mcD'(\Omega),
    \\
    (\nablas \vec{u}(x), \mat{S}(x)) &\in\Ax(x)&\quad&\text{for almost
      every}~x\in\Omega.
  \end{alignedat}
\end{align}
The symmetric part of the gradient is defined by $\nablas
\vec{u}\definedas \frac12(\nabla \vec{u}+(\nabla \vec{u})^{\rm T})$.
As is implied by the choice of the solution space $W_0^{1,r}(\Omega) ^d$
for the velocity, a homogenous Dirichlet boundary condition is assumed here for $\vecu$.

\medskip

The implicit law, which relates the shear rate $\nablas\vec{u}$ to the shear
stress $\vec{S}$, is given by an nonhomogeneous maximal monotone $r$-graph
$\Ax:x\mapsto\Ax(x)$. In particular, we assume that the following properties hold for almost every
$x\in\Omega$:
\begin{enumerate}[leftmargin=1cm,itemsep=1ex,label=(A\arabic{*})]
\item\label{A1} $(\vec{0},\vec{0})\in\Ax(x)$;
\item\label{A2} For all $(\vec{\delta}_1,\vec{\sigma}_1),(\vec{\delta}_2,\vec{\sigma}_2)\in\Ax(x)$,
  \begin{align*}
    (\vec{\sigma}_1-\vec{\sigma}_2):(\vec{\delta}_1-\vec{\delta}_2)\geq
    0 \qquad\text{($\Ax(x)$ is a monotone graph)},
  \end{align*}
  and if $\vec{\delta}_1\neq\vec{\delta}_2$ and $\vec{\sigma}_1\neq
  \vec{\sigma}_2$, then the inequality is strict;
\item\label{A3} If
  $(\vec{\delta},\vec{\sigma})\in\Rdd_{\text{sym}}\times\Rdd_{\text{sym}}$
  and
  \begin{align*}
    (\bar{\vec{\sigma}}-\vec{\sigma}):(\bar{\vec{\delta}}-
    \vec{\delta})\geq 0\quad\text{for all}~
    (\bar{\vec{\delta}},\bar{\vec{\sigma}})\in \Ax(x),
  \end{align*}
  then $(\vec{\delta},\vec{\sigma})\in\Ax(x)$ (i.e., $\Ax(x)$ is a
  maximal monotone graph);
\item\label{A4}There exists a constant $c>0$ and a nonnegative $m\in L^1(\Omega)$,
  such that for all $(\vec{\delta},\vec{\sigma})\in\Ax(x)$ we have
  \begin{align*}
    \vec{\sigma}:\vec{\delta}\ge c(|\vec{\delta}|^r+|\vec{\sigma}|^{r'})- m(x) \qquad\text{(i.e., $\Ax(x)$
      is an $r$-graph);}
  \end{align*}
\item\label{A5}
  \begin{enumerate}[leftmargin=1cm,label=(\roman{*})]
  \item For all $\vec{\delta}\in\Rdd_{\text{sym}}\definedas
    \{\vec{\zeta}\in\Rdd\colon \vec{\zeta}=\vec{\zeta}^{\rm T}\}$, the set
    \begin{align*}
      \{\vec{\sigma}\in\Rdds:(\vec{\delta},\vec{\sigma})\in\Ax(x)\}
    \end{align*}
    is closed;
  \item For any closed
    $\mathcal{C}\subset\Rdds$, the set
    \begin{align*}
      \big\{(x,\vec{\delta})\in\Omega\times\Rdd_{\text{sym}}:
      \text{there exists } \vec{\sigma}\in\mathcal{C},\text{ such that }
      (\vec{\delta},\vec{\sigma})\in\Ax(x)\big\}
    \end{align*}
    is measurable relative to the smallest $\sigma$-algebra
    $\mathfrak{L}(\Omega)\otimes\mathfrak{B}(\Rdd_{\text{sym}})$
    of the product of
    the   $\sigma$-algebra
    $\mathfrak{L}(\Omega)$ of Lebesgue measurable subsets of
    $\Omega$ and all Borel subsets
    $\mathfrak{B}(\Rdd_{\text{sym}})$ of $\Rdd_{\text{sym}}$.
  \end{enumerate}
\end{enumerate}
\medskip


The class of fluids described by \eqref{eq:implicit} is very general and includes not only Newtonian (Navier--Stokes) fluids ($\vec{S} = 2\mu_* \vec{D}\vec{v}$ with $\mu_*$ being a positive constant), but also standard power-law fluid models, where $\vec{S} = 2\mu_* |\vec{D}\vec{v}|^{r-2}\,\vec{D}\vec{v}$, $1 < r < \infty$, and their generalizations ($\vec{S} = 2\tilde\mu (|\vec{D}\vec{v}|^2)\, \vec{D}\vec{v}$), stress power-law fluid flow models and their generalizations of the form $\vec{D}\vec{v} = \alpha(|\vec{S}|^2)\, \vec{S}$, fluids with the viscosity depending on the shear rate and the shear stress
$\vec{S} = 2\hat\mu (|\vec{D}\vec{v}|^2, |\vec{S}|^2)\, \vec{D}\vec{v},$
as well as activated fluids, such as Bingham and Herschel--Bulkley fluids.
For further details concerning the physical background of the implicit constitutive theory we refer the reader to the papers by Rajagopal and Rajagopal \& Srinivasa \cite{Rajagopal:03,Rajagopal:06,RS08}, and the introductory sections of
Bul{{\'\i}}{\v{c}}ek, Gwiazda,  M{\'a}lek \& \'Swierczewska-Gwiazda
\cite{BulGwiMalSwi:09,BuGwMaSw11} and Bul{{\'\i}}{\v{c}}ek, M{\'a}lek \& S{\"u}li \cite{BMS12}.


It is proved in \cite{BulGwiMalSwi:09} that under the assumption
$r>\frac{2d}{d+2}$ problem \eqref{eq:implicit} has a weak solution.
The proof in \cite{BulGwiMalSwi:09} uses a sequence of
approximation spaces spanned by finite subsets of a Schauder basis of
an infinite-dimensional subspace of a Sobolev space, consisting of
exactly divergence-free functions. Since such a Schauder
basis is not explicitly available for computational purposes, here, instead,
we shall approximate~\eqref{eq:implicit} from two classes of inf-sup
stable pairs of finite element spaces. The first class contains
velocity-pressure space-pairs that do not lead to exactly
divergence-free velocity approximations.  For finite element spaces of
this kind our convergence result is restricted to $r>\frac{2d}{d+1}$.
In the case of exactly divergence-free finite element spaces for the
velocity we show that the resulting (sub)sequence of finite element
approximations converges to a weak solution of the problem for any $r>
\frac{2d}{d+2}$, as in \cite{BulGwiMalSwi:09}.

The paper is structured as follows.
In Section \ref{sec:preliminaries} we introduce the necessary
analytical tools, including Young measures and Chacon's biting lemma.
In Section \ref{s:fem_approx} we define the finite element
approximation of the problem with both discretely divergence-free and
exactly divergence-free finite element spaces for the velocity.  A key
technical tool in our analysis is a new \textit{discrete Lipschitz truncation}
technique, which can be seen as the finite element counterpart of the
Lipschitz truncation of Sobolev functions discovered by Acerbi and
Fusco \cite{AcerbiFusco:88} and further refined by Diening, M{\'a}lek, and
Steinhauer \cite{DieningMalekSteinhauer:08}; see also
\cite{DieHarHasRuz:11,BreDieFuc:12,BreDieSch:12}. The central result of the paper
is stated in Section \ref{s:main}, in Theorem \ref{t:main}, and
concerns the convergence of the finite element approximations
constructed in Section \ref{s:fem_approx}.

\section{Preliminaries}
\label{sec:preliminaries}

In this section we recall some known results and mathematical tools
from the literature.  We shall first introduce basic notations and
recall some well-known properties of Lebesgue and Sobolev
spaces.  We shall then discuss the approximation of an $x$-dependent
$r$-graph by a sequence of regular single-valued tensor fields using a graph-mollification
technique by Francfort, Murat and Tartar \cite{FraGilMur:04}
(which the authors of \cite{FraGilMur:04} attribute to Dal Maso).
We close the section
by recalling a generalization from
\cite{Gwiazda:05,GwiazdaZatorska:07} of the so-called fundamental
theorem on Young measures; cf. \cite{Ball:89}.

\subsection{Analytical framework}
\label{s:inf-sup}

Let $C(\bar\Omega)^d$ be the space of $d$-component vector-valued continuous functions on
$\bar\Omega$ and let $C_0(\Rdds)$ denote the space of continuous functions with
compact support in $\Rdds$.
For a measurable subset $\omega\subset\R^d$, we denote the classical
spaces of Lebesgue and vector-valued Sobolev functions by
$(L^s(\omega),\norm[s;\omega]{\cdot})$ and
$(W^{1,s}(\omega)^d,\norm[1,s;\omega]{\cdot})$, $s\in[1,\infty]$,
respectively.  Let $\mathcal{D}(\omega)\definedas
C^\infty_0(\omega)^d$  be the set
of infinitely many times differentiable $d$-component
vector-valued functions with
compact support in $\omega$; we denote by $\mathcal{D}'(\omega)$ the corresponding dual
space, consisting of distributions on $\omega$.
For $s\in[1,\infty)$, denote by $W^{1,s}_0(\omega)^d$ the closure
of $\mathcal{D}(\omega)$ in $W^{1,s}(\omega)^d$ and let
$W^{1,s}_{0,\divo}(\omega)^d\definedas \{\vecv\in
W^{1,s}_0(\omega)^d\colon \divo\vecv\equiv0\}$. The case $s=\infty$
has to be treated differently. We define $$W^{1,\infty}_0(\Omega)^d :=
W^{1,1}_0(\Omega)^d \cap W^{1,\infty}(\Omega)^d$$ and
$$W^{1,\infty}_{0,\divo}(\Omega)^d := W^{1,1}_{0,\divo}(\Omega)^d \cap
W^{1,\infty}(\Omega)^d.$$  Moreover, we denote the space of functions in
$L^s(\omega)$ with zero integral mean by $L^s_0(\omega)$. For
$s,s'\in(1,\infty)$ with $\frac1s+\frac1{s'}=1$,
$L^{s'}(\Omega)$ and $L^{s'}_0(\Omega)$ are the dual spaces of
$L^{s}(\Omega)$ and $L^{s}_0(\Omega)$, respectively. The dual of
$W^{1,s}_0(\Omega)^d$ is denoted by $W^{-1,s'}(\Omega)^d$.  For
$\omega=\Omega$ we omit the domain in our notation for norms; \eg, we
write $\norm[s]{\cdot}$ instead of $\norm[s,\Omega]{\cdot}$.

  \paragraph{Inf-sup condition} The inf-sup condition has a crucial role in the analysis of incompressible
  flow problems. It states that, for
  $s,s'\in(1,\infty)$ with  $\frac1s+\frac1{s'}=1$,
  there exists a constant $\alpha_s>0$ such that
  \begin{subequations}
    \begin{align}\label{eq:inf-sup}
      \sup_{0\neq \vecv\in W^{1,s}_0(\Omega)^d}\frac{\scp{\divo
          \vecv}{q}}{\norm[1,s]{\vecv}}\geq \alpha_s
      \norm[s']{q}\qquad\text{for all}~q\in L^{s'}_0(\Omega).
    \end{align}
    \new{This follows from the existence of the {\em
        Bogovski\u\i} operator $\mathfrak{B}:L^s_0(\Omega)\to
      W^{1,s}_0(\Omega)$, with}
    \begin{align}\label{eq:Bog}
      \new{\divo \mathfrak{B}h=h
        \qquad\text{and}\qquad\alpha_s\norm[{1,s}]{\mathfrak{B}h}\le
        \norm[s]{h}}
    \end{align}
  \end{subequations}
  \new{for all $s\in(1,\infty)$; compare e.g. with \cite{DieSchuRu:10,Bogovskii:79}.}

  \paragraph{Korn's inequality} According to \eqref{eq:implicit} the
  maximal monotone graph defined in \ref{A1}--\ref{A5} provides
  control over the symmetric part of the velocity gradient only. Korn's inequality implies
  that this suffices in order to control the norm of a Sobolev
  function; \ie for $s\in(1,\infty)$, there exists a $\gamma_s>0$ such that
  \begin{align}
    \label{eq:korn}
    \gamma_s \norm[1,s]{\vecv}\le\norm[s]{\nablas \vecv}\qquad
    \text{for all}~\vecv\in W^{1,s}_0(\Omega)^d;
  \end{align}
  compare, for example, with \cite{DieSchuRu:10}.

\subsection{Approximation of maximal monotone $r$-graphs}
\label{ss:approxAx}

In general an $x$-depen\-dent maximal monotone $r$-graph $\Ax$
satisfying~\ref{A1}--\ref{A5} cannot be represented in an explicit
fashion. However, based on a regularized {\em measurable selection}, it can be
approximated by a regular single-valued monotone tensor field.
Following \cite{FraGilMur:04}, there exists a mapping
$\bsS^*:\Omega\times\Rdds\to \Rdds$ (a selection) such that, for all
$\vec{\delta}\in\Rdds$,
$(\vec{\delta},\bsS^*(x,\vec{\delta}))\in\Ax(x)$ for almost every
$x\in\Omega$ and
\begin{enumerate}[leftmargin=1cm,itemsep=1ex,label=(a\arabic{*})]
\item \label{a1}$\bsS^*$ is measurable with respect to the product
  $\sigma$-algebra
  $\mathfrak{L}(\Omega)\otimes\mathfrak{B}(\Rdd_{\text{sym}})$;
\item \label{a2} For almost all $x\in\Omega$ the domain of $\bsS^*$ is
  $\Rdds$;
\item \label{a3}$\bsS^*$ is monotone, \ie for every
  $\vec{\delta}_1,\vec{\delta}_2\in\Rdds$ and almost all $x\in\Omega$,
  \begin{align}
    \label{eq:S*monotone}
    \left(\bsS^*(x,\vec{\delta}_1)-\bsS^*(x,\vec{\delta}_2)\right)
    :\left(\vec{\delta}_1-\vec{\delta}_2\right)\ge0;
  \end{align}
  \item \label{a4} For almost all $x\in\Omega$ and all
    $\vec{\delta}\in\Rdds$ the following
    growth and coercivity conditions hold:
    \begin{align}
      \label{eq:S*coerc}
      \abs{\bsS^*(x,\vec{\delta})}\le c_1
      \abs{\vec{\delta}}^{r-1}+k(x)\quad\text{and}\quad\bsS^*(x,\vec{\delta}):\vec{\delta}
      \ge c_2 \abs{\vec{\delta}}^{r}-m(x),
    \end{align}
    where $c_1,c_2>0$, and $k\in L^{r'}(\Omega)$ and $m\in L^1(\Omega)$
    are nonnegative functions.
\end{enumerate}

Let $\eta\in C_0(\Rdds)$
be a radially symmetric nonnegative function with support in the unit
ball $B_1(\boldsymbol{0})\subset \Rdds$ and $\int_{\Rdds}\eta\,\d \vec{\zeta}=1$. For
$n\in\N$ we then set $\eta^n(\vec{\zeta})=n^{d^2}\eta(n\vec{\zeta})$
and define
\begin{align}\label{eq:S^n}
  \begin{split}
    \bSn(x,\vec{\delta})&\definedas (\bsS^*\convolution
    \eta^n)(x,\vec{\delta}) =\int_\Rdds
    \bsS^*(x,\vec{\zeta})\,\eta^n(\vec{\delta}-\vec{\zeta})\,\d \vec\zeta
    =\int_\Rdds \bsS^*(x,\vec{\zeta})\,\d \mu^n_{\vec{\delta}}.
  \end{split}
\end{align}
Here, thanks to the equality $\int_\Rdds\eta^n\,\d \vec{\zeta}=1$ and the nonnegativity of $\eta$,
for each $\vec{\delta}\in\Rdds$,
$\d\mu_{\vec{\delta}}^n : = \eta^n(\vec{\delta}-\vec{\zeta})\,\d \vec\zeta$ defines a
probability measure that is absolutely continuous with respect to the Lebesgue
measure, with density $\eta^n(\vec{\delta}-(\cdot))$.

We recall the
following properties of the matrix function $\bSn$ from
\cite{GwiaMalSwier:07,BulGwiMalSwi:09,GwiazdaZatorska:07}.

\begin{lem}\label{l:Sn}
  The $x$-dependent matrix function $\bSn$, defined in~\eqref{eq:S^n}, satisfies
  \begin{align*}
    \left(\bSn(x,\vec{\delta}_1)-\bSn(x,\vec{\delta}_2)\right):
    \left(\vec{\delta}_1-\vec{\delta}_2\right) \geq0 \quad\text{for
      all }\vec{\delta}_1,\vec{\delta}_2\in\Rdds.
  \end{align*}
  Moreover, there exist constants $\tilde c_1,\tilde c_2>0$ and
  nonnegative functions $\tilde m\in L^{1}(\Omega)$, $\tilde k\in
  L^{r'}(\Omega)$ such that, uniformly in $n\in\N$, we have
  \begin{alignat*}{2}
    \abs{\bSn(x,\vec{\delta})}&\le \tilde c_1
    \abs{\vec{\delta}}^{r-1}+\tilde k(x)&\qquad &\text{for all }
    \vec{\delta}\in\Rdds,
    \\
    \bSn(x,\vec{\delta}):\vec{\delta}&\ge \tilde c_2
    \abs{\vec{\delta}}^{r}-\tilde m(x)&\qquad &\text{for all }
    \vec{\delta}\in\Rdds.
  \end{alignat*}
\end{lem}

\begin{rem}
  The selection $\bsS^\ast$ enters in the definition of the finite element
  method in the form of $\bSn$ through the Galerkin ansatz; compare
  with Section \ref{ss:Galerkin} below. The natural question is then
  how one can gain access to such a selection. In fact, in most
  physical models it appears that the selection $\bsS^\ast$ is given and
  $\Ax(x)$ is defined as the maximal monotone graph containing the
  set $\{(\bsD,\bsS^*(\bsD))\colon\bsD\in\Rdds\}$; compare with
  \cite{BulGwiMalSwi:09,GwiaMalSwier:07} and the references therein.
\end{rem}

\subsection{Weak convergence tools\label{ss:YoungMeasure}}
The result in \cite{Gwiazda:05,GwiazdaZatorska:07} extends
\!\cite{Ball:89}
 from limits of single distributed measures to
limits of general probability measures.
To this end we need to
introduce some standard notation from
the theory of Young measures. We denote by $\Mspace$ the space of
bounded Radon measures. We call $\mu:\Omega\to\Mspace$, $x\mapsto
\mu_x$, weak-$\ast$ measurable if the mapping
$x\mapsto \int_{\Rdds}h(\vec{\zeta})\,\d\mu_x(\vec{\zeta})$ is
measurable for all $h\in C_0(\Rdds)$. The associated nonnegative measure
is defined by  $\abs{\mu_x}(\mathcal{C})\definedas \mu_x^++\mu_x^-$,
via the Jordan decomposition $\mu_x=\mu^+_x-\mu^-_x$ into two bounded non-negative
measures $\mu^+_x,\mu^-_x$.
By means of the norm
$\norm[L^\infty_w(\Omega;\Mspace)]{\mu}\definedas \operatorname{ess\,
  sup}_{x\in\Omega}\int_\Rdds \d\abs{\mu_x}$ the space
$L^\infty_w(\Omega;\Mspace)$ of essentially bounded,
weak-$\ast$ measurable functions turns into a Banach space with separable
predual $L^1(\Omega,C_0(\Rdds))$.
The support of a non-negative measure is defined to be the largest
closed subset of $\Rdds$ for which every
open neighborhood of every point of the set has positive measure
and $\supp\mu_x:=\supp\mu_x^+\cup\supp\mu_x^-$.

 \begin{thm}[Young measures]\label{t:young_meas}
   Let $\Omega$ be an open and bounded subset of $\Rd$. Suppose that
   $\{\nu^j\}_{j\in\N}\subset L^\infty_w(\Omega;\Mspace)$ is such that
   $\nu_x^j$ is a probability measure on $\Rdds$ for all $j\in\N$ and
   a.e. $x\in\Omega$. Assume that $\nu^j$ converges to $\nu$ in the
   weak-$\ast$ topology of $L^\infty_w(\Omega;\Mspace)$ for some
   $\nu\in L^\infty_w(\Omega;\Mspace)$.

   Suppose further that the sequence $\{\nu^j\}_{j\in\N}$ satisfies the tightness
   condition
   \begin{align*}
     \lim_{R\to\infty} \sup_{j\in\N}\abs{\left\{x\in\Omega\colon \supp
       \nu_x^j\setminus B_R(\boldsymbol{0})\neq \emptyset\right\}}\to0,
   \end{align*}
   where $B_R(\boldsymbol{0})$ denotes the ball in $\Rdds$ with center $\boldsymbol{0}\in\Rdds$ and
   radius $R>0$.

   Then, the following statements hold:
   \begin{enumerate}[leftmargin=1cm,itemsep=1ex,label=(\roman{*})]
   \item \label{i}$\nu_x$ is a probability measure,
     \ie $\norm[\Mspace]{\nu_x}=\int_{\Rdds}{\rm d}\abs{\nu_x}=1$ a.e. in
     $\Omega$;
   \item \label{ii}for every $h\in L^\infty(\Omega;C_b(\Rdds))$,
     \begin{align*}
       \int_{\Rdds}h(x,\vec{\zeta})\,\d \nu_x^j(\vec{\zeta})\weak^\ast
       \int_{\Rdds}h(x,\vec{\zeta})\,\d \nu_x(\vec{\zeta})
       \qquad\text{weak-$\ast$ in}~L^\infty(\Omega);
     \end{align*}
   \item \label{iii}for every measurable subset $\omega\subset\Omega$ and for
     every Carath{\'e}odory function $h$ such that
     \begin{align}\label{eq:carath}
       \lim_{R\to\infty} \sup_{j\in\N}
       \int_\omega\int_{\left\{\vec{\zeta}\in\Rdds\colon
         \abs{h(x,\vec{\zeta})}>R\right\}}\abs{h(x,\vec{\zeta})}\,\d \nu^j_x(\vec{\zeta})\,\d \vec{\zeta} =0
     \end{align}
     we have that
     \begin{align*}
       \int_{\Rdds}h(x,\vec{\zeta})\,\d \nu_x^j(\vec{\zeta})\weak
       \int_{\Rdds}h(x,\vec{\zeta})\,\d\nu_x(\vec{\zeta})
       \qquad\text{weakly in}~L^1(\omega).
     \end{align*}
   \end{enumerate}

 \end{thm}

 \begin{lem}[Chacon's biting lemma]\label{l:biting}
   Let $\Omega$ be a bounded domain in $\R^d$ and let
   $\{v^n\}_{n\in\N}$ be a bounded sequence in $L^1(\Omega)$. Then,
   there exists a nonincreasing sequence of measurable subsets
   $E_j\subset\Omega$ with $\abs{E_j}\to0$ as $j\to\infty$, such that $\{v^n\}_{n\in\N}$ is
   precompact in the weak topology of $L^1(\Omega\setminus E_j)$, for each $j\in\N$.
 \end{lem}

\section{Finite Element Approximation}\label{s:fem_approx}

This section is concerned with approximating
problem~\eqref{eq:implicit} by a finite element method. To this end we
introduce a general framework covering inf-sup stable
Stokes elements, which are discretely divergence-free, as well as exactly
divergence-free finite elements for the velocity. These two classes of velocity
elements require different treatment of the convection term.  The discussion of
these, including representative examples of velocity-pressure pairs
of finite element spaces from each class, is the subject of
\sectionref{ss:fem} and \sectionref{ss:divfree-fem}.  The finite
element approximation of~\eqref{eq:implicit} is stated
in~\sectionref{ss:Galerkin}.
We close with a new Lipschitz
truncation method for finite element spaces, which plays a crucial
role in the proof of our main result, Theorem~\ref{t:main}.

\subsection{Finite element spaces}\label{ss:fem_spaces}
We consider a family $\{\Vn,\Qn\}_{n\in\N}\subset
W^{1,\infty}_0(\Omega)^{d}$ $\times L^{\infty}(\Omega)$ of pairs of
conforming finite-dimensional subspaces of
$W^{1,r}_0(\Omega)^{d}\times L^{\infty}(\Omega)$.  To be more precise,
let $\mathbb{G}:=\{\gridn\}_{n\in\N}$ be a sequence of shape-regular partitions of
$\bar\Omega$, \ie a sequence of regular finite element partitions of $\bar\Omega$
satisfying the following structural assumptions.
\begin{itemize}[leftmargin=0.5cm]
\item {\em Affine equivalence:} For every element $\elm\in\gridn$,
  $n\in\N$, there exists an invertible affine mapping
  \begin{align*}
    \boldsymbol{F}_\elm:\elm\to \hat \elm,
  \end{align*}
  where $\hat\elm$ is the closed standard reference $d$-simplex or the
  closed standard unit cube in~$\R^d$.
\item {\em Shape-regularity:} For any element $\elm\in\gridn$,
  $n\in\N$, the ratio of its diameter to the diameter of the largest
  inscribed ball is bounded, uniformly with respect to all partitions
  $\gridn$, $n\in\N$.
\end{itemize}
For a given partition $\gridn$, $n\in\N$, and certain subspaces
$\V\subset C(\bar\Omega)^d$ and $\Q\subset
L^\infty(\Omega)$, the finite element
spaces are then given by
\begin{subequations}\label{df:VQ}
  \begin{align}
    \Vn=\V(\gridn)&\definedas \left\{\vecV\in\V~\colon~
      \vecV_{|\elm}\circ\boldsymbol{F}_\elm^{-1}\in\hat{\mathbb{P}}_\V,
      ~\elm \in\gridn~\text{and}~
      \vecV_{|\partial\Omega}=0\right\},\label{df:Vn}
    \\
    \Qn=\Q(\gridn)&\definedas \left\{Q\in \Q~\colon
      ~Q_{|\elm}\circ\boldsymbol{F}_\elm^{-1}\in
      \hat{\mathbb{P}}_\Q,~\elm\in\gridn\right\},
    \label{df:Qn}
  \end{align}
\end{subequations}
where $\hat{\mathbb{P}}_\V\subset W^{1,\infty}(\hat\elm)^d$ and
$\hat{\mathbb{P}}_\Q\subset L^\infty(\hat\elm)$ are finite-dimensional
subspaces, with $\dim \hat{\mathbb{P}}_\V=\ell$ and $\dim
\hat{\mathbb{P}}_\V=\jmath$, respectively, for some
$\ell,\jmath\in\N$.  Note that $\Qn\subset L^\infty(\Omega)$ and since
$\Vn\subset C(\bar\Omega)^d$ it follows that $\Vn\subset
W^{1,\infty}_0(\Omega)^d$.  Each of the above spaces is assumed to have
a finite and locally supported basis; e.\,g. for the discrete pressure
space this means that for $n\in\N$ there exists $N_n\in\N$ such that
\begin{align*}
  \Qn=\operatorname{span}\{Q_1^n,\ldots,Q^n_{N_n}\}
\end{align*}
and for each basis function $Q_i^n$, $i=1,\ldots,N_n$, we have that if
there exists $\elm\in\gridn$ with $Q_i^n\not\equiv 0$
 on $\elm$, then
\begin{align*}
  \supp{Q_j^n}\subset
  \bigcup\left\{\elm'\in\gridn\mid \elm' \cap
    \elm\neq\emptyset\right\}\asdefined
  \Omega^n_\elm\new{\quad\text{with}\quad \Omega^n_\elm\le c\abs{\elm}}
\end{align*}
\new{for some constant $c>0$ depending on the shape-regularity of
  $\mathbb{G}$.
  The piecewise constant mesh size function $h_{\gridn}\in
L^\infty(\Omega)$ is almost everywhere in $\Omega$ defined by}
\begin{align*}
  \new{h_{\gridn}(x):=\abs{\elm}^{\frac 1d},\quad\text{if $\elm\in\gridn$ with $x\in\elm.$}}
\end{align*}
We introduce the subspace
$\Vndiv$ of  discretely divergence-free functions by
\begin{align*}
  \Vndiv\definedas \big\{\vecV\in\Vn\colon
  \scp{\divo\vecV}{Q}=0~\text{for all}~Q\in\Qn\big\}
\end{align*}
and we define
\begin{align*}
  \Qn_0\definedas \big\{Q\in\Qn:\int_\Omega Q\,\d x=0\big\}.
\end{align*}
Throughout the paper we assume that all pairs of velocity-pressure
finite element spaces possess the
following properties.
\begin{ass}[Approximability]\label{ass:dense}
  For all $s\in [1,\infty)$,
  \begin{align*}
      \inf_{\vecV\in\Vn} \| \vecv- \vecV \|_{1,s} \to 0& \qquad \text{for
    all}~\vecv\in W^{1,s}_0(\Omega)^{d},~\text{as $n\to\infty$; ~and} \\
       \inf_{Q\in \Qn} \| q- Q\|_{s} \to 0& \qquad \text{for
    all}~q\in L^{s}(\Omega),~\text{as $n\to\infty$.}
  \end{align*}
  \new{For this, a necessary condition is that the maximal  mesh size
    vanishes, i.e. we have $\norm[L^\infty(\Omega)]{h_{\gridn}}\to 0$ as $n\to\infty$.}
\end{ass}
\begin{ass}[Projector $\Pndiv$]\label{ass:Pndiv}
 For each $n\in\N$ there exists a linear projection
  operator $\Pndiv:W^{1,1}_0(\Omega)^{d}\to\Vn$ such that,
  \begin{itemize}[leftmargin=0.5cm]
  \item $\Pndiv$ preserves divergence in the dual of ${\Qn}$; \ie for any $\vecv\in
    W^{1,1}_0(\Omega)^{d}$, we have
    \begin{align*}
      \scp{\divo \vecv}{Q}&=\scp{\divo\Pndiv \vecv}{Q}\qquad\text{for
        all}~Q\in\Qn.
    \end{align*}
  \item $\Pndiv$ is locally $W^{1,1}$-stable; \ie there exists
    $c_1>0$, independent of $n$, such that
    \begin{align}\label{eq:stability}
      \fint_\elm \new{\abs{\Pndiv \vecv}}+h_{\gridn} \abs{\nabla\Pndiv \vecv}\,\d x\le c_1
      \fint_{\Omega^n_\elm}\new{\abs{\vecv}}+h_{\gridn}\abs{\nabla \vecv}\,\d x \!\!
    \end{align}
    for all $\vecv\in W^{1,1}_0(\Omega)^{d}$ and all       $\elm\in\gridn$.
    Here we have used the notation $\fint_B \cdot\,\d x \definedas
    \frac1{\abs{B}}\int_B\cdot\,\d x$ for the integral mean-value
    over a measurable set $B\subset \Rd$, $\abs{B}\neq0$.
  \end{itemize}
\end{ass}
It was shown in~\cite{BelBerDieRu:10,DieningRuzicka:07b} that the
local $W^{1,1}$-stability of~$\Pndiv$ implies its local and global
$W^{1,s}$-stability, $s \in [1,\infty]$.
In fact, by noting that the power function $t\mapsto t^s$ is convex for
$s\in[1,\infty)$, we obtain \new{for almost every $x\in\elm$, $\elm\in\gridn$,} by the equivalence of norms on
finite-dimensional spaces and standard scaling arguments, that
\begin{align*}
  \new{\abs{\Pndiv \vecv(x)}+h_{\gridn}}\abs{\nabla\Pndiv \vecv(x)} &\le \new{\norm[L^\infty(\elm)]{\Pndiv \vecv(x)}+h_{\gridn}}\norm[L^\infty(\elm)]{\nabla\Pndiv
    \vecv}
  \\
  &\le c\fint_\elm \new{\abs{\Pndiv \vecv(x)}+h_{\gridn}}\abs{\nabla\Pndiv \vecv}\,\d x
  \\
  &
  \le c \fint_{\Omega^n_\elm}\new{\abs{\vecv(x)}+h_{\gridn}}\abs{\nabla \vecv}\,\d x
  \\
  &\le
  c\bigg(\fint_{\Omega^n_\elm}\new{\abs{\vecv(x)}^s+h_{\gridn}^s}\abs{\nabla \vecv}^s\,\d
  x\bigg)^\frac{1}{s},
\end{align*}
where we have used Jensen's inequality in the last step; recall
that $\abs{\Omega^n_\elm}\le c\abs{\elm}$ with a constant
depending solely on the shape-regularity of $\mathbb{G}$.
Raising this inequality to the $s$-th power and integrating over $\elm$ yields
\begin{align*}
  \int_\elm \new{\abs{\Pndiv \vecv}^s+h_{\gridn}^s}\abs{\nabla\Pndiv \vecv}^s\,\d x
  \le c\int_{\Omega^n_\elm}\new{\abs{\vecv}^s+h_{\gridn}^s}\abs{\nabla
    \vecv}^s\,\d x.
\end{align*}
Summing
over all elements $\elm\in\gridn$ and accounting for the locally finite
overlap of patches 
yields, for any $s \in
    [1,\infty)$, that
\begin{align}\label{eq:W1s-stab}
  \norm[1,s]{\Pndiv \vecv}\le c_s \norm[1,s]{ \vecv}\qquad \text{for
    all}~\vecv\in W^{1,s}_0(\Omega)^{d},
\end{align}
with a constant $c_s>0$ independent of $n\in\N$. Note that for $s=\infty$
the inequality \eqref{eq:W1s-stab} follows from an obvious modification of the argument above.


Hence, by invoking the approximation properties of the sequence of finite element spaces
for the velocity, stated in Assumption \ref{ass:dense}, we obtain that
\begin{align}\label{eq:approxVn}
 \norm[1,s]{\vecv-\Pndiv \vecv}\rightarrow 0\qquad \text{for
   all}~\vecv\in W^{1,s}_0(\Omega)^{d},~\text{as $n\to\infty$ and
   $s\in [1,\infty)$}.
\end{align}
\new{Moreover, we have the following result in the weak topology of $W^{1,s}_0(\Omega)^{d}$.}
\begin{prop}\label{lem:Pnweak}
  \new{Let $\{\vecv_n\}_{n\in\N}\subset W^{1,s}_0(\Omega)^{d}$,
  $s\in(1,\infty)$,
  such that $\vecv_n\weak \vecv$ weakly in $W^{1,s}_0(\Omega)^{d}$ as
  $n\to\infty$. Then}
\begin{align*}
  \new{\Pndiv\vecv_n\weak \vecv\quad\text{weakly in}~W^{1,s}_0(\Omega)^{d}~\text{as
  $n\to\infty$}.}
\end{align*}
\end{prop}

\begin{proof}
  \new{Thanks to the uniform boundedness of the sequence of linear operators $\{\Pndiv\,:\, W^{1,s}_0(\Omega)^{d} \rightarrow \mathbb{V}^n \subset W^{1,s}_0(\Omega)^{d}\}_{n\in\N}$ (cf. \eqref{eq:W1s-stab}), we have
  that there exists a weakly converging subsequence of
  $\{\Pndiv\vecv_n\}_{n\in\N}$ in
  $W^{1,s}_0(\Omega)^{d}$. By the uniqueness of the weak limit, it therefore
  suffices to identify the limit of $\{\Pndiv\vecv_n\}_{n\in\N}$  in $L^s(\Omega)^{d}$.
  We deduce from the above
  considerations that }
  \begin{align*}
    \new{\norm[L^s(\Omega)]{\vecv -\Pndiv\vecv_n}}&\new{\le
      \norm[L^s(\Omega)]{\vecv
        -\Pndiv\vecv}+\norm[L^s(\Omega)]{\Pndiv(\vecv_n -\vecv)}}
    \\
    &\new{\le \norm[L^s(\Omega)]{\vecv
        -\Pndiv\vecv} + c \norm[L^s(\Omega)]{\vecv_n-\vecv}+c\norm[L^s(\Omega)]{h_{\gridn}\nabla(
    \vecv_n-\vecv)}}.
  \end{align*}
  \new{The first term vanishes because of~\eqref{eq:approxVn} and the second
  term vanishes since $\vecv_n\to\vecv$
  strongly in $L^s(\Omega)^{d}$, thanks to the compact embedding
  $W^{1,s}_0(\Omega)^{d}\hookrightarrow \hookrightarrow L^s(\Omega)^{d}$.
  The last term vanishes since
  $\norm[L^\infty(\Omega)]{h_{\gridn}}\to0$ as $n\to\infty$,
  by Assumption~\ref{ass:dense}. }
 \end{proof}

\begin{ass}[Projector $\PnQ$]\label{ass:PQ}
 For each $n\in\N$ there exists a linear projection
 operator $\PnQ:L^1(\Omega)\to\Qn$ such that, for all
 $s'\in(1,\infty)$,
 $\PnQ$ is stable. In other words, there exists a constant $\tilde c_{s'}>0$,
   independent of $n$, such that
   \begin{align*}
     \norm[s']{\PnQ q} \leq \tilde c_{s'} \norm[s']{q} \qquad \text{for
       all}~q\in L^{s'}(\Omega).
   \end{align*}
\end{ass}

The stability of $\PnQ$ and the approximation properties of $\Qn \subset L^{s'}(\Omega)$, stated in Assumption~\ref{ass:dense}, imply that $\PnQ$ satisfies
\begin{align}\label{eq:approxQn}
 \norm[s']{q-\PnQ q}\rightarrow 0,\qquad \text{as
   $n\to\infty$}\quad\text{for all}~q\in
 L^{s'}(\Omega) \text{ and } s' \in (1,\infty).
\end{align}
%

As a consequence of~\eqref{eq:inf-sup} and Assumption~\ref{ass:Pndiv}
(compare also with \eqref{eq:W1s-stab}) the following discrete
counterpart of~\eqref{eq:inf-sup} holds; see \cite{BelBerDieRu:10}.
\begin{prop}[Inf-sup stability]\label{p:Dinf-sup}For all
  $s,s'\in(1,\infty)$ with $\frac1s+\frac1{s'}=1$, there exists
  a constant $\beta_{s}>0$,  independent of $n$, such that
  \begin{align*}
    \sup_{0\neq \vecV\in\Vn}\frac{\scp{\divo
        \vecV}{Q}}{\norm[1,s]{\vecV}}\geq \beta_s \,
    \norm[s']{Q}\qquad\text{for all $Q\in \Qn_0$ and all $n\in\N$}.
  \end{align*}
\end{prop}

\new{Thanks to the above considerations,
  there is a {\em discrete Bogovski\u\i} operator,
  which admits  the following properties.
}
\begin{cor}[Discrete Bogovski{\u\i} operator]\label{c:Dbogovskii}
  Under the conditions of this section, \new{for all $n\in\N$, there
    exists a linear $\Bogn:\divo\Vn\to\Vn$ with}
  \begin{align*}
    \divo (\Bogn H)=H\qquad\text{and}\qquad \beta_s\,\norm[1,s]{\Bogn
      H}\leq \sup_{Q\in\Qn}\frac{\scp{H}{Q}}{\norm[s']{Q}}
  \end{align*}
  for all $H\in \divo\Vn$. \new{Moreover, if $\vecV^n\in\Vn$,
  $n\in\N$, such that $\vecV^n\weak \vecV$
  weakly in $W^{1,s}_0(\Omega)^d$ as $n\to\infty$, then we have that
  \begin{align*}
    \Bogn\divo\vecV^n \weak \mathfrak{B}\divo\vecV\quad\text{weakly in}~W^{1,s}_0(\Omega)^{d}~\text{as $n\to\infty$.}
  \end{align*}
  }
\end{cor}

It follows from Corollary~\ref{c:Dbogovskii} by H{\"o}lder's inequality
that $\beta_s\,\norm[1,s]{\Bogn H}\leq \norm[s]{H}$. However, we shall
need in the proof of Lemma~\ref{l:bn->0} the stronger statement from
Corollary~\ref{c:Dbogovskii}.

\begin{proof}
  \new{Thanks to the discrete inf-sup stability
  (Proposition~\ref{p:Dinf-sup}), we may identify $\divo\Vn$ with the
  dual of $\Qn/\R$. Next, we extend $H\in\divo\Vn$,
to $h_{H}=L_0^{s}(\Omega)$, $s\in(1,\infty)$,
by means of  the projection
operator $\Pi_\Q^n:L^{s'}(\Omega)\to\Qn$, $\frac1s+\frac1{s'}=1$, from
Assumption~\ref{ass:PQ}. In fact, $h_{H}\in L_0^s(\Omega)$ is
uniquely defined by
\begin{align*}
   \int_\Omega H\,\Pi_\Q^n q\,\dx=\int_\Omega h_{H}
   q\,\dx\quad\text{for all}~q\in L^{s'}(\Omega).
\end{align*}
Moreover, we have
\begin{align*}
  \norm[s]{h_{H}}&=\sup_{q\in L^{s'}(\Omega)}\frac{\int_\Omega
    h_{H} q\,\dx}{\norm[{s'}]{q}}
    =\sup_{q\in L^{s'}(\Omega)}\frac{\int_\Omega
   H  \,\Pi_\Q^n q\,\dx}{\norm[{s'}]{q}}
 \le
 \tilde c_{s'}\sup_{Q\in \Qn} \frac{\int_\Omega
   H   Q\,\dx}{\norm[{s'}]{Q}}.
\end{align*}
We define $\Bogn H :=\Pndiv \mathfrak{B} h_{H}\in\Vn$.
Thanks to the above considerations and the stability properties
\eqref{eq:W1s-stab} and \eqref{eq:Bog}
of $\Pndiv$ and $\mathfrak{B}$ respectively, we have proved the first
claim.

In order to prove the second assertion, we set $H^n:=\divo \vecV^n$
and conclude that $H^n\weak H:=\divo\vecV$ weakly in $L^s_0(\Omega)$
as $n\to\infty$. Consequently, thanks to~\eqref{eq:approxQn}, we have for all $q\in L^{s'}(\Omega)$, that
\begin{align*}
  \int_\Omega h_{H^n}q\,\dx = \int_\Omega H^n\PnQ q\,\dx \to \int_\Omega H q\,\dx\quad\text{as}~n\to\infty.
\end{align*}
In other words, we have that $h_{H^n}\weak H$ weakly in
$L^s_0(\Omega)$ as $n\to\infty$. The Bogovski\u i operator
$\mathfrak{B}:L_0^{s}(\Omega)\to W^{1,s}_0(\Omega)^d$ is continuous and
therefore it is also continuous with respect to
the weak topologies of the respective
spaces; compare e.g. with \cite[Theorem
6.17]{AliprantisBorder:06}. Therefore, we have
$\mathfrak{B}h_{H^n}\weak \mathfrak{B} H$ weakly in $W^{1,s}_0(\Omega)^d$
as $n\to\infty$ and the assertion follows from Proposition~\ref{lem:Pnweak}.}
\end{proof}

\subsection{Discretely divergence-free
  finite elements}
\label{ss:fem}
As in \cite{Temam:84} we wish to ensure that the discrete counterpart
of the convection term inherits the skew-symmetry of the convection
term.  In particular, upon integration by parts, it follows that
\begin{align}\label{eq:skew_sym}
  -\int_\Omega (\vecv\otimes\vecw):\nabla \vec{h} \,\d x = \int_\Omega
  (\vecv\otimes\vec{h}):\nabla\vecw+(\divo\vecv)(\vecw\cdot\vec{h})
  \,\d x
\end{align}
for all $\vecv,\vecw,\vec{h}\in
\mcD(\Omega)^d$. The last term vanishes provided  that $\divo
\vecv\equiv0$, and then
\begin{align*}
  \int_\Omega (\vecv\otimes\vecv):\nabla\vecv \,\d x=0.
\end{align*}
It can be easily seen that this is not generally true for
finite element functions
$\vecV\in\Vn$, even if
\begin{align*}
    \scp{\divo\vecV}{Q} &=0\qquad\text{for all}~Q\in\Qn,
\end{align*}
\ie if $\vec{V}$ is discretely divergence-free.  However, we observe
from~\eqref{eq:skew_sym} that
\begin{align}\label{eq:trilin}
  \begin{split}
    -\int_\Omega (\vecv\otimes\vecw):\nabla\vec{h}
    \,\d x={\textstyle \frac12}\int_\Omega
    (\vecv\otimes\vec{h}):\nabla \vecw
    -(\vecv\otimes\vecw):\nabla \vec{h}
   \,\d x
    \asdefined
      \Trilin{\vecv}{\vecw}{\vec{h}}
  \end{split}
\end{align}
for all $\vecv,\vecw,\vec{h}\in W^{1,\infty}_{0,\divo}(\Omega)^d$.
 We extend this definition to $W^{1,\infty}(\Omega)^d$
in the obvious way and deduce that
\begin{align}
  \label{eq:SKEW_SYM}
  \Trilin{\vecv}{\vecv}{\vecv}=0\qquad\text{for all }\vecv\in
  W^{1,\infty}(\Omega)^d.
\end{align}

We further investigate this modified convection term for fixed
$r,r'\in(1,\infty)$ with $\frac1r+\frac1{r'}=1$; recall the definition
of $\tr$ from \eqref{eq:tr}.  We note that $\tr>1$ is equivalent to
the condition $r>\frac{2d}{d+2}$. In this case we can define its dual
$\tr'\in(1,\infty)$ by $\frac1\tr+\frac1{\tr'}=1$ and we note that the
Sobolev embedding
\begin{align}\label{eq:tr-emb}
  W^{1,r}(\Omega)^d\hookrightarrow L^{2\tr}(\Omega)^d
\end{align}
holds.  This is a crucial condition in the continuous problem, which
guarantees
\begin{align}\label{eq:cont_div0}
  \int_\Omega (\vecv\otimes\vecw):\nabla \vec{h}\,\d x\le c\,
  \norm[1,r]{ \vecv}\norm[1,r]{\vecw}\norm[1,\tr']{\vec{h}}
\end{align}
for all $\vecv,\vecw,\vec{h}\in
W^{1,\infty}(\Omega)^d$;
see~\cite{BulGwiMalSwi:09} and Section~\ref{ss:divfree-fem} below.
Because of the extension \eqref{eq:trilin} of the convection term to
functions that are not necessarily point-wise divergence-free, we have to adopt
the following stronger condition in order to ensure that the trilinear
form $\Tri[\cdot,\cdot,\cdot]$ is bounded on $W^{1,r}(\Omega)^d\times W^{1,r}(\Omega)^d\times
W^{1,\tr'}(\Omega)^d$. In particular,
let $r>\frac{2d}{d+1}$, in order to ensure that there exists
$s\in(1,\infty)$ such that $\frac{1}{r}+\frac{1}{2\tr}+\frac1{s}=1$.
In other words, we have for $\vecv,\vecw,\vec{h}\in
W^{1,\infty}(\Omega)^d$ that
\begin{align*}
  \int_\Omega(\divo\vecv)\,(\vecw\cdot\vec{h})\,\d x\le \norm[r]{\divo
    \vecv}\norm[2\tr]{\vecw}\norm[s]{\vec{h}} \le c\,
  \norm[1,r]{ \vecv}\norm[1,r]{\vecw}\norm[1,\tr']{\vec{h}},
\end{align*}
with a constant $c$ depending on $r$, $\Omega$ and $d$.  Here we have used
the embeddings \eqref{eq:tr-emb} and
$W^{1,\tr'}_0(\Omega)^d\hookrightarrow L^s(\Omega)^d$.
Consequently, together
with \eqref{eq:cont_div0} we thus obtain
\begin{align}
  \label{eq:cont_trilin}
  \Trilin{\vecv}{\vecw}{\vec{h}}\le c\,
  \norm[1,r]{\vecv}\norm[1,r]{\vecw}\norm[1,\tr']{\vec{h}}.
\end{align}

\begin{ex}\label{ex:dfree-fe}
  In \cite{BelBerDieRu:10} it is shown that Assumptions
  \ref{ass:Pndiv} and \ref{ass:PQ} are satisfied by the following
  velocity-pressure pairs of finite elements:
  \begin{itemize}[leftmargin=0.5cm]
  \item The conforming Crouzeix--Raviart Stokes element, \ie continuous
    piecewise quadratic plus bubble velocity and discontinuous
    piecewise linear pressure approximations (compare e.\,g. with
    \cite[\S VI Example 3.6]{BrezziFortin:91});
  \item The Mini element; see,  \cite[\S VI Example
    3.7]{BrezziFortin:91};
  \item The spaces of continuous piecewise quadratic elements for the
    velocity and piecewise constants for the pressure (\cite[\S VI
    Example 3.6]{BrezziFortin:91});
  \end{itemize}
  Moreover, it is stated without proof in \cite{BelBerDieRu:10} that
  the lowest order Taylor--Hood element also satisfies Assumptions
  \ref{ass:Pndiv} and \ref{ass:PQ}.
\end{ex}

\subsection{Exactly divergence-free finite elements}
\label{ss:divfree-fem}

Another way of retaining the skew-symmetry of the convection term and
ensuring that \eqref{eq:SKEW_SYM} holds is to use an exactly
divergence-free finite element approximation of the velocity.  In
addition to Assumptions~\ref{ass:Pndiv} and~\ref{ass:PQ} in
Section~\ref{ss:fem_spaces} we suppose that the following condition
holds.
\begin{ass}\label{ass:divfree}
  The finite element spaces defined in
  Section~\ref{ss:fem_spaces} satisfy
  \begin{align*}
    \divo \Vn \subset \Qn_0,\qquad n\in\N.
  \end{align*}
  This inclusion obviously implies that discretely divergence-free
  functions are automatically exactly divergence-free, \ie
  \begin{align*}
    \Vndiv=\{\vecV\in\Vn\colon \divo\vecV\equiv0\},\qquad n\in\N.
  \end{align*}
\end{ass}
According to \eqref{eq:skew_sym}, in this case, we define
\begin{align}
  \label{eq:trilin2}
  \Trilin{\vecv}{\vecw}{\vec{h}}\definedas -\int_\Omega(\vecv\otimes
  \vecw):\nabla \vec{h}\,\d x
\end{align}
for all $\vecv,\vecw,\vec{h}\in W^{1,\infty}_0(\Omega)^d$ and obtain
\begin{align}
  \label{eq:SKEW_SYM2}
  \Trilin{\vecv}{\vecv}{\vecv}=0 \qquad\text{for all }\vecv\in
  W^{1,\infty}_{0,\divo}(\Omega)^d.
\end{align}
Recalling \eqref{eq:cont_div0}, with Assumption \ref{ass:divfree}, the
convection term can be
controlled under the weaker restriction  $r>\frac{2d}{d+2}$, \ie
for
$\vecv,\vecw,\vec{h}\in W^{1,\infty}(\Omega)^d$, we have that
\begin{align}
  \label{eq:cont_trilin2}
  \Trilin{\vecv}{\vecw}{\vec{h}}\leq c\,
  \norm[1,r]{\vecv}\norm[1,r]{\vecw}\norm[1,\tr']{\vec{h}},
\end{align}
where, as before, $\frac1\tr+\frac1{\tr'}=1$ with $\tr$ from
\eqref{eq:tr}. The constant $c>0$ only depends on $r$, $\Omega$
and $d$.

Admittedly, finite
element spaces that simultaneously  satisfy Assumptions~\ref{ass:Pndiv},~\ref{ass:PQ}
and~\ref{ass:divfree} are not very common. Most constructions of
exactly divergence-free finite element spaces in the literature
 are not very practical in that they require a sufficiently high polynomial
degree and/or restrictions on the geometry
of the mesh; see~\cite{ArnoldQuin:92,ScottVogelius:85,QinZhang:07,Zhang:08}. In a very
recent work~\cite{GuzmanNeilan:2011}
Guzm{\'a}n and Neilan proposed inf-sup stable finite element pairs in two space-dimensions,
which admit exactly divergence-free
velocity approximations for $r=2$. A generalization of the Guzm{\'a}n--Neilan elements to three
dimensions is contained in~\cite{GuzmanNeilan:2012}.
We shall show below that the lowest order spaces introduced in~\cite{GuzmanNeilan:2011}
simultaneously satisfy Assumptions~\ref{ass:Pndiv},~\ref{ass:PQ}
and \ref{ass:divfree} for $r\in[1,\infty)$.

\begin{ex}[Guzm{\'a}n--Neilan elements \cite{GuzmanNeilan:2011}]
  \label{ex:GuzmanNeilan}
  We consider the finite element spaces introduced by Guzm{\'a}n and
  Neilan in \cite[Section 3]{GuzmanNeilan:2011} on simplicial
  triangulations of a bounded open polygonal domain~$\Omega$ in $\mathbb{R}^2$. In particular, we define
  \begin{align*}
    \hat{\mathbb{P}}_\V\definedas \mathbb{P}_1(\hat\elm) ~\oplus~
    \operatorname{span}\{\curl(\hat b_i): i=1,2,3\}
    ~\oplus~
    \operatorname{span}\{\curl(\hat B_i): i=1,2,3\}.
  \end{align*}
  Here $\mathbb{P}_1(\hat\elm)$ denotes the space of affine vector-valued
  functions over $\hat\elm$. Let further $\{\hat\lambda_i\}_{i=1,2,3}$ be the
  barycentric coordinates on $\hat E$ associated with
  the three vertices $\{\hat  z_i\}_{i=1,2,3}$ of $\hat\elm$, \ie
  $\hat\lambda_i(\hat z_j)=\delta_{ij}$. Then, for $i\in\{1,2,3\}$, we set
  $\hat b_i\definedas \hat\lambda_{i+1}^2\hat\lambda_{i+2}$,
  and $\hat B_i$ denotes the rational bubble
  function \begin{align*}
    \hat B_i\definedas \frac{\hat\lambda_i\hat\lambda_{i+1}^2
      \hat\lambda_{i+2}^2}{(\hat\lambda_i
      +\hat\lambda_{i+1})(\hat\lambda_i+\hat\lambda_{i+2})},
  \end{align*}
  which can be continuously extended by zero at $\hat{z}_{i+1}$ and
  $\hat{z}_{i+2}$; the index $i$ has to be understood modulo $3$.
  Thanks to properties of the $\curl$ operator, the local pressure
  space
  \begin{align*}
    \hat{\mathbb{P}}_\Q\definedas \divo \hat{\mathbb{P}}_\V
  \end{align*}
  is the space of constant functions over $\hat\elm$.

  It is clear from \cite{GuzmanNeilan:2011} that the related pairs
  of spaces $\{\Vn,\Qn\}_{n\in\N}$ (compare with \eqref{df:VQ})
  satisfy Assumption \ref{ass:divfree}. For $\PnQ$ we can use a
  Cl{\'e}ment type interpolation or simply the
  best-approximation in $\Qn$; clearly, both satisfy Assumption
  \ref{ass:PQ}. The approximability assumption, Assumption \ref{ass:dense}, follows with
  the mesh-size tending
  to zero.
  It remains to verify Assumption \ref{ass:Pndiv}. To
  this end we analyze the interpolation operator proposed in
  \cite{GuzmanNeilan:2011}. In particular, let
  $\Pi_S^n:W^{1,1}_0(\Omega)^2\to \mathbb{L}^n$ be the Scott--Zhang
  interpolant \cite{ScottZhang:90} into the linear Lagrange finite
  element space $\mathbb{L}^n\definedas\mathbb{L}(\gridn)$ over a triangulation
  $\gridn$, belonging to a shape-regular family of triangulations $\mathbb{G}=\{\gridn\}_{n \in \mathbb{N}}$ of $\Omega$. Then
  $\Pndiv:W^{1,1}_0(\Omega)^2\to \Vn$ is defined by
  \begin{align}\label{eq:ipol}
    \begin{alignedat}{2}
      (\Pndiv \vecv)(z)&\definedas (\Pi_S^n \vecv)(z),&\qquad
      z&\in\nodes^n,
      \\
      \int_\side
      \Pndiv\vecv\,\d s&=\int_\side \vecv\,\d s,&\qquad\side&\in\sides^n,
    \end{alignedat}
  \end{align}
  where $\nodes^n$ and $\sides^n$ denote the vertices, respectively edges,
  of the triangulation $\gridn$, $n\in\N$. This operator is a projector and
  thanks to \cite[(3.14)]{GuzmanNeilan:2011} and the fact that
  $\mathbb{L}^n\subset \Vn$ it thus remains to prove the stability
  estimate \eqref{eq:stability} in Assumption \ref{ass:Pndiv}.

  To this end we fix $n\in\N$. Although the
  claim can be proved using the techniques in
  \cite{GuzmanNeilan:2011}, this would necessitate the introduction of
  additional notation. Thus, for the sake of brevity of the presentation, we give an
  alternative proof.  According to
  \cite{GuzmanNeilan:2011} the interpolation operator
  $\Pndiv$ is correctly defined by \eqref{eq:ipol}. Let
  $\{\hat z_i\}_{i=1,2,3}$ be the vertices of $\hat\elm$ and let
  $\{\hat S_i\}_{i=1,2,3}$ be its edges. Then, any function
  $\hat\vecV\in\hat{\mathbb{P}}_\V$ is uniquely defined by
  $\hat \vecV(\hat z_i)$ and $\int_{\hat\side_i}\hat\vecV\,\d s$,
  $i=1,2,3$. This implies that the mapping
  \begin{align*}
    \hat \vecV \mapsto \sum_{i=1}^3 \big|\hat \vecV(\hat
    z_i)\big|+\bigg|\int_{\hat\side_i}\hat\vecV \,\d s\, \bigg|,\qquad\hat\vecV\in\hat{\mathbb{P}}_\V,
  \end{align*}
  is a norm on $\hat{\mathbb{P}}_\V$. Hence, equivalence of norms on
  finite-dimensional spaces together with \eqref{eq:ipol} yield that
  \begin{align*}
    \int_{\hat\elm}\abs{\Pndiv \vecv\circ
      \boldsymbol{F}_\elm^{-1}}\,\d x&\le c  \sum_{i=1}^3 \Big(\big|\Pndiv
    \vecv\circ \boldsymbol{F}_\elm^{-1}(\hat
    z_i)\big|+\big|\int_{\hat\side_i}\Pndiv \vecv\circ
    \boldsymbol{F}_\elm^{-1} \,\d s \big|\Big)
    \\
    &= c  \sum_{i=1}^3 \Big(\big|\Pi^n_S
    \vecv\circ \boldsymbol{F}_\elm^{-1}(\hat
    z_i)\big|+\big|\int_{\hat\side_i}\vecv\circ
    \boldsymbol{F}_\elm^{-1} \,\d s \big|\Big),\quad\vecv\in W^{1,1}_0(\Omega)^2.
  \end{align*}
  By a scaled trace theorem
  and properties of the Scott--Zhang operator we arrive at
  \begin{align*}
    \int_\elm \abs{\Pndiv \vecv}\,\d x \le c\int_{\Omega^n_\elm}\abs{
      \vecv}+h_{\gridn}\abs{\nabla \vecv}\,\d x,\qquad \vecv\in W^{1,1}_0(\Omega)^2.
  \end{align*}
  Note that $\Pndiv:W^{1,1}_0(\Omega)^2\to\Vn$ is a projector and
  that $\mathbb{L}^n\subset\Vn$. Thus the inequality~\eqref{eq:stability} follows from a standard inverse
  estimate and the Bramble--Hilbert Lemma; compare also with
  \cite[Theorem 3.5]{BelBerDieRu:10}.

\end{ex}

\subsection{The Galerkin approximation\label{ss:Galerkin}}
We are now ready to state the discrete problem. Let
$\{\Vn,\Qn\}_{n\in\N}$ be the finite element spaces of Section
\ref{ss:fem} or \ref{ss:divfree-fem} and let
$\Tri:W^{1,\infty}_0(\Omega)^d\times W^{1,\infty}_0(\Omega)^d\times
W^{1,\infty}_0(\Omega)^d\to\R$ be defined correspondingly.

For $n\in\N$ we call a
triple of
functions $\big(\Un,\,P^n,\,\bSn(\nablas \Un)\big)\in\Vn\times \Qn_0\times
L^{r'}(\Omega)^{d\times d}$ a Galerkin approximation of
\eqref{eq:implicit} if it satisfies
\begin{align}\label{eq:discrete}
   \begin{aligned}
      \int_\Omega\bSn(\cdot,\nablas\Un):\nablas \vecV \,\d x
      +\Trilin{\Un}{\Un}{\vecV}-\scp{\divo \vecV}{P^n}
      &=\dual{\vecf}{\vecV}
      \\
      &\qquad\quad\text{for all}~ \vecV\in\Vn,
      \\
      \int_\Omega Q\divo\Un\,\d x
      &=0\hspace{2pt}\quad\text{for all}~Q\in\Qn.
    \end{aligned}
\end{align}

Restricting the test-functions to $\Vndiv$ the discrete problem
\eqref{eq:discrete} reduces to finding $\Un\in\Vndiv$ such that
\begin{align}\label{eq:Vn0}
  \int_\Omega\bSn(\cdot,\nablas\Un):\nablas \vecV\,\d x
      +\Trilin{\Un}{\Un}{\vecV}= \dual{\vecf}{
        \vecV}\qquad\text{for all}~\vecV\in\Vndiv.
\end{align}
Thanks to~\eqref{eq:SKEW_SYM}, respectively~\eqref{eq:SKEW_SYM2},
it follows from Lemma~\ref{l:Sn}
and Korn's inequality~\eqref{eq:korn} that the nonlinear operator
defined by the left-hand side of \eqref{eq:Vn0} is coercive
on $\Vndiv$.
Since the dimension of $\Vndiv$ is finite, Brouwer's fixed point
theorem ensures the existence of a solution to~\eqref{eq:Vn0}. The
existence of a solution triple to~\eqref{eq:discrete} then follows by
the discrete inf-sup stability implied by Proposition~\ref{p:Dinf-sup}.

Of course, because of the
weak assumptions in the definition of the maximal monotone $r$-graph,
\eqref{eq:discrete} does not define the Galerkin approximation $\Un$
uniquely.
However for each $n\in\N$ we may select an arbitrary one among possibly
infinitely many solution triples and thus obtain a sequence
\begin{align}\label{eq:SEQ}
  \big\{\big(\Un,P^n,\bSn(\cdot,\nablas\Un) \big) \big\}_{n\in\N}.
\end{align}

\subsection{Discrete Lipschitz truncation}
\label{ss:Lipschitz}
In this section we shall present a modified Lipschitz truncation,
which acts on finite element spaces. This {\em discrete Lipschitz
  truncation} is basically a composition of a continuous Lipschitz
truncation and the projector from Assumption \ref{ass:Pndiv}.
For this reason we first introduce a new Lipschitz truncation on
$W^{1,1}_0(\Omega)^d$, based on the results in
\cite{DieningMalekSteinhauer:08,BreDieFuc:12,BreDieSch:12},
which provides finer estimates than
the original Lipschitz truncation technique proposed by Acerbi and
Fusco in \cite{AcerbiFusco:88}.

For $v\in L^{1}(\R^d)$ we define the Hardy--Littlewood
maximal function
\begin{align}\label{df:M}
  M(v)(x)&\definedas
  \sup_{R>0}\fint_{B_R(x)}\abs{v}\,\d y.
\end{align}
For $s\in(1,\infty]$ the
Hardy--Littlewood maximal operator $M$ is continuous from
$L^s(\Rd)$ to $L^{s}(\Rd)$, \ie there exists a constant $c_s>0$ such that
\begin{align}
  \label{eq:HLmax}
  \norm[L^s(\Rd)]{M(v)} \leq c_s\,\norm[L^s(\Rd)]{v}
  \qquad\text{for all } v\in L^s(\Rd),
\end{align}
and it is of weak type $(1,1)$, \ie there exists a constant $c_1>0$ such that
\begin{align}
  \label{eq:HLmax_weak}
  \sup_{\lambda>0}\,\lambda \abs{\{M(v)>\lambda\}} \leq c_1\norm[L^1(\Rd)]{v}
  \qquad\text{for all } v\in L^1(\Rd);
\end{align}
see, e.\,g., \cite{Gra:04}.  For any $\vecv\in W^{1,1}(\R^d)^d$ we set
$M(\vecv)\definedas M(\abs{\vecv})$ and $M(\nabla \vecv)\definedas
M(\abs{\nabla\vecv})$.

Let $\vecv\in W^{1,1}_0(\Omega)^d$; we may then assume that $\vecv\in
W^{1,1}(\Rd)^d$ by extending $\vecv$ by zero outside $\Omega$.  For fixed
$\lambda>0$ we define
\begin{subequations}\label{df:UH}
  \begin{align}
    \label{eq:U}
    \begin{split}
      \mcU_\lambda(\vecv)&\definedas \left\{M(\nabla\vecv)>
        \lambda\right\},
    \end{split}
    \intertext{and} \label{df:H} \mcH_{\lambda}(\vecv)&\definedas \R^d
    \setminus (\mcU_\lambda(\vecv) \cap \Omega) = \left\{ M(\nabla\vecv)\le
      \lambda\right\} \cup (\R^d \setminus \Omega).
  \end{align}
\end{subequations}
Since $M(\nabla \vecv)$ is lower semi-continuous, the set
$\mcU_\lambda(\vecv)$ is open and the set $\mcH_\lambda(\vecv)$ is closed.
According to~\cite{DieningMalekSteinhauer:08} it follows that $\vecv$
restricted to $\mcH_\lambda(\vec{v})$ is Lipschitz continuous and
therefore also bounded. More precisely, we have that
\begin{align}\label{eq:Lcont}
  \abs{\vecv(x)-\vecv(y)}\leq c\,\lambda\abs{x-y}
\end{align}
for all $x,y\in \mcH_{\lambda}(\vec{v})$, where the constant $c>0$
depends on $\Omega$.

It remains to extend $\vecv_{|\mcH_{\lambda}(\vec{v})}$ to a Lipschitz
continuous function on $\R^d$.
The result in~\cite{DieningMalekSteinhauer:08} is based on the
so-called Kirszbraun extension theorem (cf. Theorem 2.10.43 in \cite{Federer:69})
and uses an additional truncation of $\vecv$ with respect to
$M(\vec{v})$. This can be avoided by proceeding similarly as
in~\cite{BreDieFuc:12,BreDieSch:12}, \ie
extending $\vecv_{|\mcH_{\lambda}(\vec{v})}$
by means of a partition of unity on a
Whitney covering of the open and bounded set $ \mcU_\lambda(\vecv)$.
To this end, we assume w.l.o.g. that $\mcU_\lambda(\vecv)
\neq\emptyset$; otherwise $\vecv$ does not need to be extended since
$\mcH_\lambda(\vecv)=\R^d$.
According to \cite{Gra:04,BreDieFuc:12} there exists a decomposition
of the open set
$\mcU_\lambda(\vecv)$ into a family of (dyadic) closed cubes
$\{\mcQ_j\}_{j\in\N}$, with side lengths
$\ell_j\definedas\ell(\mcQ_j)$, $j\in\N$, such that
\begin{enumerate}[leftmargin=1cm,itemsep=1ex,label={(W\arabic{*})}]
\item \label{itm:Wdis} $\bigcup_{j\in\N} {\mcQ_j} = \mcU_\lambda(\vecv)$
  and the $\mcQ_j$'s have pair-wise disjoint interiors.
\item \label{itm:Wbnd} $8 \sqrt{d}\, \ell(\mcQ_j) \leq {\rm dist}(\mcQ_j,
  \partial \mcU_\lambda(\vecv)) \leq 32 \sqrt{d}\, \ell(\mcQ_j)$.
\item \label{itm:W3}If $\mcQ_j\cap\mcQ_k\neq\emptyset$ for some $j,k\in\N$, then
  \begin{align*}
    \textstyle{\frac{1}{2}} \leq \frac{\ell_j}{\ell_k} \leq 2.
  \end{align*}
\item \label{itm:Wnumber} For a given $\mcQ_j$ there exist at most $(3^d-1)2^d$ cubes $\mcQ_k$ with $\mcQ_j\cap\mcQ_k\neq\emptyset$.
\end{enumerate}
For a fixed cube $\mcQ\in\Rd$ with barycenter $z$ and any $c>0$, we define
\begin{align*}
  c\mcQ\definedas\left\{x\in\Rd\colon
    \max_{i=1,\ldots,d}\abs{x_i-z_i}\leq c\,\ell(\mcQ)\right\}.
\end{align*}
Hence, it follows from \ref{itm:Wbnd} with
$\theta_d := 2+64 \sqrt{d}$, that
\begin{align}
  (\theta_d \mcQ_j) \cap \mcH_\lambda(\vecv) \neq \emptyset\quad
  \text{for all $j\in\N$.}\label{eq:theta_d}
\end{align}

Let $\mcQ_j^\ast := \sqrt{\frac{9}{8}} \mcQ_j$ and $\mcQ_j^{\ast\ast} :=
\frac{9}{8} \mcQ_j$. Thanks to \ref{itm:Wnumber}, the enlarged cubes
$\mcQ_j^{\ast\ast}$, $j\in\N$, are locally finite, i.e., they satisfy $\sum_j
\chi_{\mcQ_j^{\ast\ast}} \leq c$ with a constant $c>0$, which depends
on $d$ only. Thanks to the overlaps of the $\mcQ^\ast_j$'s, there
exists a partition of unity $\{\psi_j\}_{j\in\N}$ subordinated to the
family~$\{\mcQ_j^\ast\}_{j\in\N}$ with the following properties:
\begin{itemize}
\item $\sum_j \psi_j =\chi_{\mcU_\lambda(\vec{v})}$ and $0 \leq \psi_j
  \leq 1$ for all $j\in\N$;
\item $\chi_{\frac 78 \mcQ_j} \leq\psi_j \leq \chi_{\mcQ_j^\ast}$, for
  all $j\in\N$;
\item $\psi_j \in C^\infty_0(\mcQ_j^\ast)$ with $\abs{\nabla \psi_j}
  \leq c\,\ell_j^{-1}$, for all $j\in\N$.
\end{itemize}

The Lipschitz truncation of $\vec{v}$ is then denoted by $\vec{v}_\lambda$ and is defined by
\begin{subequations}
  \label{eq:lip}
  \begin{align}
    \label{eq:lip1}
    \vec{v}_\lambda &:=
    \begin{cases}
      \sum_{j\in \N} \psi_j \vec{v}_j &\qquad \text{in
        $\mcU_\lambda(\vec{v})$},
      \\
      \vec{v} &\qquad \text{elsewhere},
    \end{cases}
  \intertext{where}
\label{eq:lip2}
    \vec{v}_j &:=
    \begin{cases}
      \fint_{\mcQ_j^{\ast\ast}} \vec{v}\,\dx &\qquad \text{if } \mcQ_j^\ast
      \subset \Omega,
      \\
      0 &\qquad \text{elsewhere}.
    \end{cases}
  \end{align}
\end{subequations}
We emphasize that the definition of the functions $\vecv_j$, $j \in \mathbb{N}$, here differs from the one
in~\cite{BreDieFuc:12}, since we need to preserve the no-slip boundary
condition for the velocity field on~$\partial \Omega$ under Lipschitz truncation.
Combining the techniques of~\cite{DieningMalekSteinhauer:08} and~\cite{BreDieFuc:12}
we obtain the following result; for ease of readability of the main body of the paper, the proof
of Theorem \ref{thm:Tlest} is deferred to the Appendix.
\begin{thm}
  \label{thm:Tlest}
  Let $\lambda>0$ and $\vecv \in W^{1,1}_0(\Omega)^d$. Then, the Lipschitz
  truncation defined in \eqref{eq:lip} has the following properties:
  $\vecv_\lambda \in W^{1,\infty}_0(\Omega)^d$, and
  \begin{enumerate}[leftmargin=1cm,itemsep=1ex,label={\rm (\alph{*})}]
  \item \label{itm:Tlestsupp} $\vecv_\lambda= \vecv$ on
    $\mcH_\lambda(\vecv)$, \ie $\{\vecv\neq\vecv_\lambda\}\subset
    \mcU_\lambda(\vecv) \cap \Omega = \{ M(\nabla
    \vec{v})>\lambda\}\cap\Omega$;
  \item \label{itm:TlestL1} $\norm[s]{\vecv_\lambda} \leq
    c\,\norm[s]{\vecv}$ for all $ s \in [1,\infty]$, with $\vecv
    \in L^s(\Omega)^d$;
  \item \label{itm:TlestW1h} $\norm[s]{\nabla \vecv_\lambda}
    \leq c\, \norm[s]{\nabla \vecv}$ for all $s \in [1,\infty]$,
    with $\vecv \in W^{1,s}_0(\Omega)^d$;
  \item \label{itm:TlestLinfty} $|\nabla \vecv_\lambda| \leq c\,
    \lambda \chi_{\mcU_\lambda(\vecv) \cap \Omega} + \abs{\nabla
      \vecv} \chi_{\mcH_\lambda(\vecv)} \leq c \lambda$ almost
    everywhere in $\Rd$.
  \end{enumerate}
  The constants appearing in the inequalities stated in parts
  \ref{itm:TlestL1}, \ref{itm:TlestW1h} and \ref{itm:TlestLinfty}
  depend on $\Omega$ and $d$. In \ref{itm:TlestL1} and \ref{itm:TlestW1h} they
  additionally depend on $s$.
\end{thm}

We next modify the Lipschitz truncation so that for finite
element functions the truncation is again a finite element function.
To this end we recall the definition of the finite
element space $\Vn=\V(\gridn)$ of Section~\ref{ss:fem_spaces} or
\ref{ss:divfree-fem}.

Let $\lambda>0$ and fix $n\in\N$. Since $\Vn\subset W^{1,1}_0(\Omega)^d$, we could
apply the truncation defined in \eqref{eq:lip}. However,
since in general the Lipschitz truncation ${\vecV}_{\!\lambda}$ of $\vec{V}\in\Vn$ does not
belong to $\Vn$, we shall define the discrete Lipschitz truncation by
\begin{align}\label{eq:Vnl}
  \vecV_{n,\lambda}\definedas \Pndiv  \vecV_{\lambda}\in\Vn.
\end{align}

According to the next lemma the interpolation operator $\Pndiv$ is local, in the sense that it
modifies $\vecV_\lambda$ in a neighborhood of $\mcU_{\lambda}(\vecV)$ only.
  \begin{lem}\label{l:SO}
    Let  $\vecV\in\Vn$. With the notation adopted in this section, we have that
    \begin{align*}
      \!\!\!\!\!\left\{\vecV_{n,\lambda}\neq
       \vecV\right\}\subset\Omega_{\lambda}^n(\vecV)\definedas
     \operatorname{interior}\left(\bigcup\big\{\Omega^n_\elm\mid\elm\in\gridn~\text{with}~
      \elm\cap \mcU_{\lambda}(\vecV)\neq\emptyset \big\}\right).
   \end{align*}
 \end{lem}

 \begin{proof}
   The stated inclusion follows from \eqref{eq:stability} in Assumption
   \ref{ass:Pndiv}. In particular, let
   $\elm\in \gridn$ be such that $\elm\subset \Rd\setminus
   \Omega_\lambda^n(\vecV)$; then,
   $\Omega^n_\elm\subset\mcH_\lambda(\vecV)$.
   Consequently, by Theorem \ref{thm:Tlest}\ref{itm:Tlestsupp}, we have
   $\vecV_{\lambda}=\vecV$ on $\Omega^n_\elm$. Hence we deduce
   from
   \eqref{eq:stability}, our assumption that $\vecV\in\Vn$,  and the fact that $\Pndiv$
   is a projector, that
   \begin{align*}
     \fint_\elm \abs{\vecV-\Pndiv\vecV_{\lambda}}\,\d x&=
     \fint_\elm \abs{\Pndiv(\vecV-\vecV_{\lambda})}\,\d x
     \\
     &\le
     c
     \fint_{\Omega^n_\elm}
     \abs{\vecV-\vecV_{\lambda}}+h_{\gridn}\abs{\nabla(\vecV-\vecV_{\lambda})
     }\,\d x=0,
   \end{align*}
   \ie $\vecV=\vecV_{n,\lambda}=\Pndiv\vecV_{\lambda}$ on $\elm$.
   This proves the assertion.
   %
 \end{proof}

 The set $\Omega^n_\lambda(\vecV)$ from Lemma \ref{l:SO}
 is larger than $\mcU_\lambda(\vecV) \cap \Omega$. However, the next result
 states that we can keep the increase of the set under control.
 This is the key observation for the construction of the discrete Lipschitz truncation.
 \begin{lem}\label{l:OsubU}
   For $n\in\N$, $\vecV\in\Vn$ and $\lambda>0$, let
   $\Omega_\lambda^n(\vecV)$ be defined as in Lemma~\ref{l:SO}. Then,
   there exists a
   $\kappa\in(0,1)$ only depending on $\hat\P_\V$ and
   the shape-regularity of $\mathbb{G}$, such that
   \begin{align*}
     \mcU_\lambda(\vecV) \cap \Omega \subset
     \Omega_\lambda^n(\vecV) \subset
     \mcU_{\kappa\lambda}(\vecV) \cap \Omega.
   \end{align*}
 \end{lem}
 \begin{proof}
   Thanks to the definition of $\Omega_\lambda^n(\vecV)$, the first inclusion is clear.
   It thus remains to show the second inclusion.
   In order
   to avoid problems at the boundary $\partial\Omega$ we extend
   $\vecV$ to $\Rd$ by zero outside $\Omega$.
   Let
   $x\in\Omega_\lambda^n(\vecV)$; then, there
   exists $\elm\in\gridn$, $\elm\cap\mcU_\lambda(\vecV)\neq\emptyset$
   such that $x\in\Omega^n_\elm$. Consequently, by \eqref{eq:U} and \eqref{df:M},
   there exists an $x_0\in\elm$ and an $R>0$
   such that
   \begin{align*}
     \fint_{B_R(x_0)}\abs{\nabla\vecV}\,\d y>\lambda.
   \end{align*}

  Suppose that $B_R(x_0)\subset
  \left(\Omega^n_\elm\cup(\Rd\setminus\Omega)\right)$;
  then, thanks to norm-equivalence
  in finite-dimensional spaces, we have, by a standard scaling
  argument, that
  \begin{align*}
    \lambda<\fint_{B_R(x_0)}\abs{\nabla \vecV}\,\d y\le
    \norm[L^{\infty}(\Omega^n_\elm)]{\nabla \vecV}\le \tilde{c}_1
    \fint_{\Omega^n_\elm}\abs{\nabla \vecV}\,\d y,
  \end{align*}
  where the constant $\tilde{c}_1$ depends solely on $\hat\P_\V$ and the
  shape-regularity of $\mathbb{G}$. Let $B_\rho(x)$ be the
  smallest ball with $\overline{B_\rho(x)} \supset \Omega^n_E$ and
  observe that
  $\abs{B_{\rho}(x)}\le \tilde{c}_2 \abs{\Omega^n_\elm}$ with a
  constant $\tilde{c}_2>0$ depending only on the shape-regularity of
  $\mathbb{G}$.  Consequently,
  \begin{align*}
    M(\nabla \vecV)(x) \geq \fint_{B_\rho(x)} \abs{\nabla \vecV}\,\dy
    \geq \frac{1}{\tilde{c}_2} \fint_{\Omega^n_E}\abs{\nabla \vecV}\,\dy
    \geq \frac{1}{\tilde{c}_1 \tilde{c}_2} \lambda.
  \end{align*}
   In other words, we have that $x\in \mcU_{(\tilde{c}_1\tilde{c}_2)^{-1}
    \lambda}(\vecV) \cap \Omega$.

  We now consider the case $B_R(x_0)\nsubset
  \left(\Omega^n_\elm\cup(\Rd\setminus\Omega)\right)$.  Since
  $x_0\in\elm$, it follows that $\tilde{c}_3R\ge \diam(\elm)$, with a
  constant
  $\tilde{c}_3>0$ only depending on the shape-regularity of $\mathbb{G}$.
  As $x\in\Omega^n_\elm$, there exists a constant $\tilde{c}_4>1$ such that
  $B_{\tilde{c}_4R}(x)\supset(\elm\cup B_{R}(x_0))$. Hence,
  \begin{align*}
    M (\nabla \vecV)(x)\geq \fint_{B_{\tilde{c}_4R}(x)}\abs{\nabla
      \vecV}\,\dy\geq \tilde{c}_4^{-d} \fint_{B_R(x_0)}\abs{\nabla
      \vecV}\,\dy > \tilde{c}_4^{-d} \lambda,
  \end{align*}
  and we deduce that $x\in \mcU_{\tilde{c}_4^{-d} \lambda}(\vecV) \cap
  \Omega$. Combining the two cases, the claim follows with $\kappa:=
  \min \{ (\tilde{c}_1\tilde{c}_2)^{-1}, \tilde{c}_4^{-d}\}$.
\end{proof}


We are now ready to state the following analogue
of Theorem~\ref{thm:Tlest} for the discrete Lipschitz truncation \eqref{eq:Vnl}.
\begin{thm}
  \label{thm:dTlest}
  Let $\lambda>0$, $n\in\N$ and $\vecV \in \Vn$. Then, the Lipschitz
  truncation defined in \eqref{eq:Vnl} satisfies
  $\vecV_{n,\lambda} \in \Vn$, and the following statements hold:
    \begin{enumerate}[leftmargin=1cm,itemsep=1ex,label={\rm (\alph{*})}]
  \item \label{itm:dTlestsupp} $\vec{V}_{n,\lambda}= \vec{V}$ on $\R^d
    \setminus \Omega^n_\lambda(\vec{V})$;
  \item \label{itm:dTlestW1h} $\norm[1,s]{\vecV_{n,\lambda}}
    \leq c\, \norm[1,s]{\vecV}$ for $1< s \leq \infty$;
  \item \label{itm:dTlestLinfty} $|\nabla \vecV_{n,\lambda}| \leq c\,
    \lambda \chi_{\Omega^n_\lambda(\vecV)} + \abs{\nabla \vecV}
    \chi_{\R^d \setminus \Omega^n_\lambda(\vecV)} \leq c\, \lambda$
    almost everywhere in $\Rd$.
  \end{enumerate}
  The constants $c$ appearing in the inequalities in parts \ref{itm:dTlestW1h} and \ref{itm:dTlestLinfty}
  depend on $\Omega$, $d$, $\hat\P_\V$ and the
  shape-regularity of $\mathbb{G}$. In \ref{itm:dTlestW1h} the constant $c$
  also depends on~$s$.
\end{thm}

\begin{proof}
  Assertion \ref{itm:dTlestsupp} is proved in Lemma \ref{l:OsubU}.
  Estimate~\ref{itm:dTlestW1h} is a consequence of
  Theorem~\ref{thm:Tlest} and the $W^{1,q}$-stability of~$\Pndiv$;
  compare with~\eqref{eq:W1s-stab}. The bound \ref{itm:dTlestLinfty}
  follows from Theorem~\ref{thm:Tlest}\ref{itm:TlestLinfty}
  and the
  $W^{1,\infty}$ stability of~$\Pndiv$; see \eqref{eq:W1s-stab}.
\end{proof}

The following corollary is an application of the discrete Lipschitz truncation
to (weak) null sequences. It is similar to the results
in~\cite{DieningMalekSteinhauer:08} and~\cite{BreDieFuc:12}.
Its analogue for Sobolev functions is stated in
Corollary~\ref{cor:remlip} in the
Appendix.
\begin{cor}
  \label{cor:dremlip}
  Let $1< s< \infty$ and let $\{\vec{E}^n\}_{n\in\N}\subset W^{1,s}_0(\Omega)^d$ be
  a sequence, which converges to zero weakly in $W^{1,s}_0(\Omega)^d$, as $n\to\infty$.

  Then, there exists a
  sequence $\{\lambda_{n,j}\}_{n,j\in\N}\subset\R$ with $2^{2^j} \leq
  \lambda_{n,j}\leq 2^{2^{j+1}-1}$ such that the Lipschitz truncations
  $\vec{E}^{n,j} := \vec{E}^n_{n,\lambda_{n,j}}$, $n,j\in\N$, have the following
  properties:
  %
  \begin{enumerate}[leftmargin=1cm,itemsep=1ex,label={\rm (\alph{*})}]
  \item \label{itm:dremlip1} $\vec{E}^{n,j}\in \Vn$
    and $\vec{E}^{n,j}=\vec{E}^n$ on $\R^d\setminus
    \Omega^n_{\lambda_{n,j}}(\vec{E}^n)$;
  \item \label{itm:dremlip2} $\|\nabla\vec{E}^{n,j}\|_\infty\leq
    c\,\lambda_{n,j}$;
  \item  \label{itm:dremlip3} $\vec{E}^{n,j} \to 0$ in $L^\infty(\Omega)^d$
     as $n \to \infty$;
  \item \label{itm:dremlip4} $\nabla\vec{E}^{n,j} \weak^\ast 0$ in $L^\infty(\Omega)^{d\times d}$
     as $n \to \infty$;
  \item \label{itm:dremlip5} For all $n,j \in \N$ we have
    $\norm[s]{\lambda_{n,j} \chi_{
        \Omega^n_{\lambda_{n,j}}(\vec{E}^n)}} \leq c\,
    2^{-\frac{j}{s}}\norm[s]{\nabla \vecE^n}$.
  \end{enumerate}
  The constants $c$ appearing in the inequalities \ref{itm:dremlip2} and \ref{itm:dremlip5} depend on
  $d$, $\Omega$, $\hat\P_\V$ and the shape-regularity of  $\mathbb{G}$. The constant in part \ref{itm:dremlip5} also depends on $s$.
\end{cor}

\begin{proof}
  We first construct the sequence $\lambda_{n,j}$
  and prove~\ref{itm:remlip5}. Let $\kappa>0$ be the constant in Lemma
  \ref{l:OsubU}. Then, for $g \in L^s(\R^d)$,
  we have  \new{
  \begin{align*}
    \int_{\R^d} \abs{g}^s \,\dx &= \int_{\R^d} \int_0^\infty
    \kappa^ss t^{s-1} \chi_{\{\abs{g}>\kappa t\}} \,\dt \,\dx \geq
    \int_{\R^d} \sum_{m \in \N} \kappa^s 2^{ms}
    \chi_{\{\abs{g}>\kappa 2^{m+1}\}} \,\dx
    \\
    &\geq \sum_{j \in \N} \sum_{m=2^j}^{2^{j+1}-1}\kappa^s
    2^{ms} \abs{\{\abs{g}>\kappa 2^{m+1}\}}.
  \end{align*}
  We apply this estimate to $g = 2\,M(\nabla \vecE^{n})$ and use the
  boundedness of the maximal operator~$M$ (cf. \eqref{eq:HLmax}) to obtain
  \begin{align*}
    \kappa^s\sum_{j\in \N} \sum_{m=2^j}^{2^{j+1}-1} 2^{ms}
    \abs{\{M(\nabla \vecE^n)>\kappa 2^{m}\}} &\leq 2^s\,\norm[s]{M(\nabla
      \vecE^n)}^s \leq 2^s c_s \norm[s]{\nabla \vecE^n}^s.
  \end{align*}
  For fixed $n,j$ the sum over~$m$ involves $2^j$
  summands. Consequently, there exists an integer $\lambda_{n,j} \in \{2^{2^j},
  \dots, 2^{2^{j+1}-1}\}$ such that
  \begin{align*}
    \lambda_{n,j}^s \abs{\{M(\nabla \vecE^n)>\kappa \lambda_{n,j}\}} \leq
    2^{-j}\kappa^{-s}  \,2^s c_s \norm[s]{\nabla \vecE^n}^s.
  \end{align*}}
  This, together with \new{the second inclusion in} Lemma \ref{l:OsubU}, proves~\ref{itm:remlip5}.
  Assertions \ref{itm:remlip1} and~\ref{itm:remlip2} are then direct consequences of
  Theorem~\ref{thm:dTlest}\ref{itm:dTlestsupp} and
  \ref{itm:dTlestW1h}. It remains to prove~\ref{itm:remlip3}
  and~\ref{itm:remlip4}.


  To prove \ref{itm:dremlip4}, we proceed as follows. Thanks to the uniqueness of the
  limits, it suffices to
  prove that $\vecE^{n,j}\weak 0$ weakly in~$W^{1,s}_0(\Omega)^d$.
  To this end, we note that the compact embedding
  $W^{1,s}_0(\Omega)^d\hookrightarrow \hookrightarrow L^s(\Omega)^d$
  implies that
  \begin{align*}
    \vecE^n \to 0 \quad\text{in}~L^s(\Omega)^d~\text{as $n\to\infty$.}
  \end{align*}
  Let $\{\vecE^n_{\lambda_{n,j}}\}_{n\in\N}$ be the sequence of
  Lipschitz-truncated functions, defined according to \eqref{eq:lip}.
  Then, thanks to the boundedness of $\{\vecE^n\}_{n\in\N}$ in
  $W^{1,s}_0(\Omega)^d$, Theorem \ref{thm:Tlest}\ref{itm:TlestW1h} and
  \ref{itm:TlestL1},
  we have that
  \begin{align*}
    \vecE^n_{\lambda_{n,j}}\weak 0 \quad\text{weakly
      in}~W^{1,s}_0(\Omega)^d,~\text{as $n\to\infty$.}
  \end{align*}
  Thanks to the definition of the discrete Lipschitz truncation in
  \eqref{eq:Vnl}, the desired assertion
  follows \new{from Proposition~\ref{lem:Pnweak}. Moreover, using a compact
    embedding, this also proves \ref{itm:dremlip3}.}
\end{proof}



\section{The main theorem}
\label{s:main}

After the preceding considerations, we are now ready to state our main
result. Its proof is presented in subsections \sectionref{ss:basic_conv}--\sectionref{ss:young_meas}.
\begin{thm}\label{t:main}
  Let $\{\Vn,\Qn\}_{n\in\N}$ be the sequence of finite element space
  pairs from Section~\ref{ss:fem} (respectively
  \ref{ss:divfree-fem}) and let
  $\big\{\big(\Un,P^n,\bSn(\cdot,\nablas\Un) \big) \big\}_{n\in\N}$ be
  the sequence of discrete solution triples to~\eqref{eq:discrete}
  constructed in~\eqref{eq:SEQ}.

  If $r>\frac{2d}{d+1}$ (respectively $r>\frac{2d}{d+2}$), then there
  exists a solution $(\vecu,p,\bsS) \in W^{1,r}_0(\Omega)^d\times
  L^{\tr}_0(\Omega)\times L^{r'}(\Omega)^{d\times d}$
  of~\eqref{eq:implicit}, such that, for a (not relabeled) subsequence,
  we have
  \begin{alignat*}{2}
    \Un &\weak \vecu&\qquad& \text{weakly in }
    W^{1,r}_0(\Omega)^d,
    \\
    P^n&\weak p&\qquad &\text{weakly in }
    L^{\tr}_0(\Omega),
    \\
    \bSn(\cdot,\nablas \Un)&\weak \bsS &\qquad &\text{weakly in }
    L^{r'}(\Omega)^{d\times d}.
  \end{alignat*}
\end{thm}

\subsection{Convergence of the finite element approximations}
\label{ss:basic_conv}

We begin the proof of Theorem \ref{t:main} by showing the existence of
a weak limit for the sequence of solution triples.
\begin{lem}
  \label{l:basic_conv}
  Let $\{\Vn,\Qn\}_{n\in\N}$ be the sequence of finite element space
  pairs from Section~\ref{ss:fem} (respectively
  \ref{ss:divfree-fem}) and let
  $\big\{\big(\Un,P^n,\bSn(\cdot,\nablas\Un) \big) \big\}_{n\in\N}$ be
  the sequence of discrete solution triples to~\eqref{eq:discrete}
  constructed in~\eqref{eq:SEQ}.

  If $r>\frac{2d}{d+1}$ (respectively $r>\frac{2d}{d+2}$), then there
  exists $(\vecu,p,\bsS) \in W^{1,r}_0(\Omega)^d\times
  L^{\tr}_0(\Omega)\times L^{r'}(\Omega)^{d\times d}$, such that, for
  a (not relabeled) subsequence, we have
  \begin{alignat*}{2}
    \Un &\weak \vecu&\qquad& \text{weakly in }
    W^{1,r}_0(\Omega)^d,
    \\
    P^n&\weak p&\qquad &\text{weakly in }
    L^{\tr}_0(\Omega),
    \\
    \bSn(\cdot,\bsD \Un)&\weak \bsS &\qquad &\text{weakly in }
    L^{r'}(\Omega)^{d\times d}.
  \end{alignat*}
  Moreover, the triple $(\vecu,p,\bsS)$ satisfies
  \begin{align}\label{eq:bS}
    \begin{alignedat}{2}
      \int_\Omega\bsS:\nablas \vecv - (\vecu\otimes\vecu):\nabla
      \vecv\,\d x-\scp{\divo\vecv}{p}&=\dual{\vecf}{\vecv},&~&\text{
        for all}~\vecv\in W^{1,\tr'}_0(\Omega)^d
      \\
      \int_\Omega q\divo \vecu\,\d x&=0,&~&\text{ for all}~q\in
      L^{r'}(\Omega).
    \end{alignedat}
  \end{align}
\end{lem}

\begin{proof}
  We divide the proof into four steps.

  \textbf{Step 1:} From~\eqref{eq:discrete} we see that $\Un$ is
  discretely divergence-free and thus, thanks to~\eqref{eq:Vn0}
  and~\eqref{eq:SKEW_SYM} (respectively~\eqref{eq:SKEW_SYM2}), we have  that
  \begin{align*}
    \int_\Omega\bSn(\cdot,\nablas \Un):\nablas \Un\,\d x=
    \dual{\vecf}{\Un}\leq \norm[-1,r']{\vecf}\norm[1,r]{\Un}.
  \end{align*}
  The coercivity of $\bSn$ (Lemma~\ref{l:Sn}) and Korn's inequality
  \eqref{eq:korn}
  imply
  that the sequence $\{\Un\}_{n\in\N}\subset W^{1,r}_0(\Omega)^d$ is
  bounded, independent of $n\in\N$. This in turn implies,
  again by Lemma~\ref{l:Sn}, the boundedness of
  $\{\bSn(\nablas\Un)\}_{n\in\N}$ in $L^{r'}(\Omega)^{d\times
    d}$. In other words, there exists a constant $c>0$ such that
  \begin{align}\label{eq:bound}
    \norm[1,r]{\Un}+ \norm[r']{\bSn(\cdot,\nablas \Un)}\leq c,\qquad
    \text{for all}~ n\in\N.
  \end{align}
  As $r\in(1,\infty)$, the spaces $W^{1,r}_0(\Omega)^d$ and
  $L^{r'}(\Omega)^{d\times d}$ are reflexive and thus for a (not
  relabeled) subsequence
  there exist $\vecu\in W^{1,r}_0(\Omega)^d$ and  $\bsS\in
  L^{r'}(\Omega)^{d\times d}$, such that
  \begin{align}
    \Un &\weak \vecu\qquad\text{weakly in }
    W^{1,r}_0(\Omega)^d\label{eq:weakUn} \intertext{and}
    \bSn(\cdot,\nablas \Un)&\weak\bsS\qquad\text{weakly in }
    L^{r'}(\Omega)^{d\times d},\label{eq:weakSn}
  \end{align}
  as      $n\to\infty$. Moreover, using compact embeddings of Sobolev
  spaces, we have that
  \begin{align}\label{eq:strongUn}
    \Un\rightarrow \vecu\qquad\text{strongly in}\quad
    L^s(\Omega)^d\quad\text{for all}\quad
    \begin{cases}
      s\in\left(1,\tfrac{rd}{d-r}\right), \quad&\text{if}~ r<d,
      \\
      s\in(1,\infty),\quad&\text{otherwise}.
    \end{cases}
  \end{align}
  Thanks to~\eqref{eq:weakUn} we have by \eqref{eq:approxQn},
  for arbitrary $q\in L^{r'}(\Omega)$, that
  \begin{align}\label{eq:divu=0}
    0=\int_\Omega (\PnQ q)\divo \Un\,\d x\rightarrow
    \int_\Omega q\divo\vecu\,\d x,
  \end{align}
  \ie
  the function $\vecu\in W^{1,r}_0(\Omega)^d$ is exactly
  divergence-free.

  \textbf{Step 2:}
  Next, we investigate the convection term. Let $\vecv\in
  W^{1,\infty}_0(\Omega)^d$ be arbitrary
  and define $\vecVn\definedas \Pndiv\vecv$. 
  We show that
  \begin{align}\label{eq:Bn->B}
    \Trilin{\Un}{\Un}{\vecVn}\to
    -\int_\Omega(\vecu\otimes\vecu):\nabla \vecv\,\d x.
  \end{align}
  Thanks to the assumption~$r>\frac{2d}{d+2}$ and~\eqref{eq:strongUn}, it follows that
   \begin{align*}
    \Un\otimes\Un&\to \vecu\otimes\vecu \qquad \text{in} \quad
    L^{s}(\Omega)^{d\times d}\quad\text{for all } s\in[1,\tr),
  \end{align*}
  with $\tr>1$ as in~\eqref{eq:tr}. By \eqref{eq:approxVn}, we
  have that
  $\vecVn\to\vecv$ in $W^{1,s'}_0(\Omega)^d$, $s'\in
  (\tr',\infty)$, and hence we obtain that, as $n \rightarrow \infty$,
  \begin{align}\label{eq:B1}
    -\int_\Omega (\Un\otimes\Un): \nabla \vecVn\,\d x \to
    -\int_\Omega (\vecu\otimes\vecu):\nabla \vecv\,\d x.
  \end{align}
  This proves \eqref{eq:Bn->B} for the exactly divergence-free
  approximations from Section \ref{ss:divfree-fem}. We emphasize that we have only required
  so far that $r>\frac{2d}{d+2}$.

  In order to prove \eqref{eq:Bn->B} for the finite element spaces of
  Section \ref{ss:fem} and thus for the modified convection term
  defined in \eqref{eq:trilin}, we recall from~\eqref{eq:skew_sym}
  that
  \begin{align*}
    \int_\Omega (\Un\otimes\vecVn): \nabla \Un\,\d x &=-\int_\Omega
    (\Un\otimes\Un): \nabla \vecVn +(\divo\Un)\,\Un\cdot\vecVn\,\d x.
  \end{align*}
  For the first term we have already shown convergence
  in~\eqref{eq:B1}. In view of the definition of $\Tri$
  in~\eqref{eq:trilin} it thus remains to prove that the second term
  vanishes in the limit $n\to\infty$. To this end, we observe
  by~\eqref{eq:strongUn} and Assumption~\ref{ass:Pndiv} that
  \begin{align*}
    \Un\cdot\vecVn\rightarrow \vecu \cdot \vecv\quad\text{strongly
      in}\quad L^{s}(\Omega)\quad\text{for all}\quad
    \begin{cases}
      s\in\left(1,\tfrac{rd}{d-r}\right), \quad&\text{if}~ r<d,
      \\
      s\in(1,\infty),\quad&\text{otherwise}.
    \end{cases}
  \end{align*}
  Thanks to the stronger restriction $r>\frac{2d}{d+1}$ now, this last statement holds in particular
  for $s=r'$. Hence, together with~\eqref{eq:weakUn}
  and~\eqref{eq:divu=0}, we deduce that
  \begin{align*}
    \int_\Omega(\divo\Un)\,\Un\cdot\vecVn\,\d x\to
    \int_\Omega(\divo\vecu)\,\vecu\cdot\vecv\,\d x =0
  \end{align*}
  as $n\to\infty$.

  \textbf{Step 3:} We combine the above results.  Recall that
  by~\eqref{eq:divu=0} we have $\divo\vecu\equiv0$, which is the
  second equation in~\eqref{eq:bS}.  For an arbitrary $\vecv\in
  W^{1,\infty}_{0,\divo}(\Omega)^d$ let $\vecVn\definedas \Pndiv \vecv$,
  $n\in\N$. Thanks to \eqref{eq:approxVn}, we have that
  $\vecVn\in\Vndiv$ and $\vecVn\to \vecv$ in $W^{1,s}_0(\Omega)^d$ for
  all $s\in(1,\infty)$.  Therefore,
  using~\eqref{eq:weakUn},~\eqref{eq:weakSn} and~\eqref{eq:Bn->B}, we
  obtain
  \begin{alignat*}{2}
    \int_\Omega\vec{S}^n(\cdot,\nablas\Un):\nablas \vecVn\,\d x
    \hspace{-5mm}&\hspace{5mm}+ \Trilin{\Un}{\Un}{\vecVn}&
    \quad=\quad\dual{\vecf}{\vecVn}\hspace{-10mm}&\hspace{10mm}
    \\
    &\downarrow&&\downarrow
    \\
    \int_\Omega\bsS:\nablas \vecv&+ \divo
    (\vecu\otimes\vecu)\cdot\vecv\,\d x&
    &\hspace{-3mm}\dual{\vecf}{\vecv}
  \end{alignat*}
  as $n\to\infty$.  Since $\bsS\in L^{r'}(\Omega)^{d\times d}$,
  $\vecf\in W^{-1,r'}(\Omega)^d$ and $\vecu\otimes\vecu\in
  L^{\tr}(\Omega)^{d\times d}$, by a density argument, we arrive at
  \begin{align*}
    \int_\Omega\bsS:\nablas \vecv&+ \divo
    (\vecu\otimes\vecu)\cdot\vecv\,\d x =\dual{\vecf}{\vecv}
  \end{align*}
  for all $\vecv\in W^{1,\tr'}_{0,\divo}(\Omega)^d$.

  \textbf{Step 4:} We now prove convergence of the pressure. Thanks to
  the restriction $r>\frac{2d}{d+1}$ we have, as
  in~\eqref{eq:cont_trilin}, that
  \begin{align*}
    \scp{\divo
      \vecV}{P^n}&=\int_\Omega\bSn(\cdot,\nablas\Un):\nablas\vecV\,\d
    x+\Trilin{\Un}{\Un}{\vecV} -\dual{\vecf}{\vecV}
    \\
    &\le \norm[r']{\bSn(\cdot,\nablas\Un)}\norm[r]{\nablas\vecV}+
    c\,\norm[1,r]{\Un}^2\norm[1,\tr']{\vecV} +
    \norm[-1,r']{\vecf}\norm[1,r]{\vecV}
  \end{align*}
  for all $\vecV\in\Vn$.  By~\eqref{eq:bound} and the discrete inf-sup
  condition stated in Proposition~\ref{p:Dinf-sup}, it follows that the sequence
  $\{P^n\}_{n\in\N}$ is bounded in the reflexive Banach space
  $L^{\tr}_0(\Omega)$. Hence, there exists $p\in
  L^{\tr}_0(\Omega)$ such that, for a (not relabeled) subsequence,
  $P^n\weak  p$ weakly in $L^{\tr}_0(\Omega)$.  On the other
  hand we deduce for an arbitrary $\vecv\in W^{1,\infty}_0(\Omega)^d$
  that
  \begin{align*}
    \scp{\divo \vecv}{P^n}&=\scp{\divo\Pndiv\vecv}{P^n}+
    \scp{\divo(\vecv-\Pndiv\vecv)}{P^n}
    \\
    &= \int_\Omega\bSn(\cdot,\nablas\Un):\nablas \Pndiv\vecv\,\d
    x-\dual{\vecf}{\Pndiv\vecv} +\Trilin{\Un}{\Un}{\Pndiv \vecv}
    \\
    &\qquad+\scp{\divo(\vecv -\Pndiv\vecv)}{P^n}
    \\
    &\to \int_\Omega \bsS:\nablas\vecv+
    \divo(\vecu\otimes\vecu)\cdot\vecv\,\d x-\dual{\vecf}{\vecv} +0
  \end{align*}
  as $n\to\infty$, where we have
  used~\eqref{eq:weakSn},~\eqref{eq:Bn->B},~\eqref{eq:approxVn} and the boundedness of the sequence
  $\{P^n\}_{n\in\N}$ in~$L^{\tr}_0(\Omega)$.
  This completes the proof of the lemma.
\end{proof}

For the main result, Theorem~\ref{t:main}, it  remains to prove that
\begin{align}
  \label{eq:SDv}
  \big(\nablas \vecu(x),\,\bsS(x)\big)\in\Ax(x)
\end{align}
for almost every $x\in\Omega$. The proof of this is the subject of the rest of
Section \ref{s:main}.

\subsection{Identification of the limits\label{ss:limits}}

In this section we shall first briefly discuss properties of the
maximal monotone $r$-graph introduced in~\ref{A1}--\ref{A5}. Here we
follow the presentation in \cite{BulGwiMalSwi:09}. Application of the
fundamental theorem on Young measures (cf. Theorem~\ref{t:young_meas})
leads to a representation of weak limits, which is a crucial step in
proving~\eqref{eq:SDv}.  \medskip

According to \cite{FraGilMur:04} there exists a function
$\phi:\Omega\times\Rdds$ such that
\begin{align}\label{eq:AxChi}
  \Ax(x)=\left\{(\vec{\delta},\vec{\sigma})\in\Rdds\times\Rdds:
    \vec{\sigma}-\vec{\delta}=\phi(x,\vec{\sigma}+\vec{\delta})
  \right\},
\end{align}
and
\begin{enumerate}[leftmargin=1cm,itemsep=1ex,label=(\alph{*})]
\item $\phi(x,\vec{0})=\vec{0}$ for almost every $x\in\Omega$;
\item $\phi(\cdot,\vec{\chi})$ is measurable for all $\vec{\chi}\in
  \Rdds$;
\item for almost all $x\in\Omega$ the mapping $\phi(x,\cdot)$ is
  1-Lipschitz continuous;
\item the functions $\bss,\bsd:\Omega\times\Rdds\to\Rdds$, defined as
  \begin{align}\label{df:s&d}
    \bss(x,\vec{\chi})\definedas {\textstyle\frac12} \left(\vec{\chi}
      +\phi(x,\vec{\chi})\right),\quad \bsd(x,\vec{\chi})\definedas
    {\textstyle\frac12} \left(\vec{\chi} -\phi(x,\vec{\chi})\right)
  \end{align}
  satisfy, for almost every $x\in\Omega$ and all $\vec{\chi}\in\Rdds$,
  the estimate
  \begin{align*}
    \bss(x,\vec{\chi}):\bsd(x,\vec{\chi})\ge -m(x) +c\left(\abs{
        \bsd(x,\vec{\chi})}^{r}+\abs{\bss(x,\vec{\chi})}^{r'}\right).
  \end{align*}
\end{enumerate}
We emphasize that this is in fact a characterization of maximal monotone
$r$-graphs $\Ax$ satisfying~\ref{A1}--\ref{A5} without the second part
of~\ref{A2}.

We recall the selection $\bsS^\ast$ from Section~\ref{ss:approxAx} and, as
in \cite{BulGwiMalSwi:09}, we define
\begin{align}
  \label{df:bn}
  b_n(x)\definedas
  \int_{\Rdds}\big(\bsS^\ast(x,\vec{\zeta})-\bsS^\ast(x,\nablas
  \vecu)\big):(\vec{\zeta}-\nablas \vecu(x))\,\d \mu^n_x(\vec{\zeta}),
\end{align}
where we have used the abbreviation $\mu_x^n\definedas
\mu^n_{\nablas\Un(x)}$. The next result, whose proof is postponed to the next section, states that $b_n$
vanishes in measure.
\begin{lem}\label{l:bn->0}
  With the definitions of this section we have that
    $b_n\to 0$ 
    in measure.
\end{lem}

\smallskip

Actually, employing the above
characterization of $\Ax$,
the limit of the sequence $\{b_n\}_{n \in \mathbb{N}}$ can be identified in another way by using
the theory of Young measures. To this end we introduce
\begin{align}\label{eq:Gx}
  \Gx(\vec{\zeta})\definedas \bsS^\ast(x,\vec{\zeta})+\vec{\zeta}, \qquad
  x\in\Omega,\; \vec{\zeta}\in\Rdds,
\end{align}
and define the push-forward measure of the
measure
$\mu_{\vec{\zeta}}^n$ from~\eqref{eq:S^n} by setting
\begin{align}\label{df:nun}
  \nu_{x,\vec{\zeta}}^n(\mathcal{C})=\mu_{\vec{\zeta}}^n
  \big(\Gx^{-1}(\mathcal{C})\big)\qquad \text{for all } \mathcal{C}\in
  \mathfrak{B}(\Rdds\big).
\end{align}
We recall from \sectionref{ss:approxAx}\ref{a1} that $\bsS^\ast$ is
measurable with respect to the product $\sigma$-algebra
$\mathfrak{L}(\Omega)\otimes\mathfrak{B}(\Rdd_{\text{sym}})$, and
therefore so is $\Gx$. Consequently, the measure
$\nu_{x,\vec{\zeta}}^n$ is well-defined and, thanks to properties of
the mollifier $\eta^n$ from Section~\ref{ss:approxAx}, it is a
probability measure.  From the definitions of the functions $\bss$ and
$\bsd$ it follows that
$\bsS^\ast(x,\vec{\zeta})=\bss(x,\Gx(\vec{\zeta}))$ and
$\vec{\zeta}=\bsd(x,\Gx(\vec{\zeta}))$. We thus have, by simple
substitution, the identities
\begin{subequations}\label{df:SUb}
  \begin{align}\label{df:SUba}
     \bSn(x,\nablas
      \Un)&=\int_{\Rdds}\bss(x,\vec{\zeta})\,\d \nu_{x}^n(\vec{\zeta}),
      \\\label{df:SUbb} \nablas \Un(x)&=\int_{\Rdds}
      \bsd(x,\vec{\zeta})\,\d \nu_{x}^n(\vec{\zeta}),
    \end{align}
    as well as
    \begin{gather}
    \label{df:SUbc}
    b_n(x)=\int_{\Rdds}\big(\bss(x,\vec{\zeta})-
    \bsS^\ast(x,\nablas\vecu(x))\big)
    :(\bsd(x,\vec{\zeta})-\nablas\vecu(x)) \,\d
    \nu_{x}^n(\vec{\zeta}),
  \end{gather}
\end{subequations}
where we have used the abbreviation $\nu_{x}^n\definedas
\nu_{x,\nablas\Un(x)}^n$.

In order to identify the limit we apply the generalized version of the
classical fundamental theorem on Young measures stated in Theorem
\ref{t:young_meas}.  Recall from Section \ref{ss:YoungMeasure} that
$L^\infty_w(\Omega;\Mspace)$ is a separable Banach space with predual
$L^1(\Omega,C_0(\Rdds))$.  For every $n\in\N$ the mapping $x\mapsto
\nu_x^n$ belongs to $L^\infty_w(\Omega;\Mspace)$. To see this let
$g\in L^1(\Omega;C_0(\Rdds))$. In view of \eqref{df:nun}
and~\eqref{eq:S^n}, a simple substitution yields
\begin{align*}
  \int_{\Rdds}g(x,\vec{\zeta})\,\d \nu_x^n(\vec{\zeta})
  &=\int_{\Rdds}\eta^n(\nablas\Un(x)-\vec{\zeta})\,g(x,\Gx(\vec{\zeta}))\,\d \vec{\zeta}.
\end{align*}
It remains to prove that $x\mapsto
\int_{\Rdds}\eta^n(\nablas\Un(x)-\vec{\zeta})\,g(x,\Gx(\vec{\zeta}))\,\d
\vec{\zeta}$ is measurable and integrable. It follows from~\ref{A5},
the definition~\eqref{eq:Gx} of $\Gx$ and property~\ref{a1} of
$\bsS^\ast$ that $h$ is
$\mathfrak{L}(\Omega)\otimes\mathfrak{B}(\Rdd_{\text{sym}})$
measurable. Moreover, $\eta^n$, $n\ge1$, are smooth functions and
$g\in L^1(\Omega,C_0(\Rdds))$, and therefore integrability follows
from Fubini's theorem.

Thanks to the properties of $\eta^n$ it is clear that $\nu_x^n$ is a
probability measure a.e. in~$\Omega$. Hence
$\norm[L^\infty_w(\Omega;\Mspace)]{\nu_x^n}=1$ and thus the sequence
$\{\nu^n\}_{n\in\N}$ is bounded in $L^\infty_w(\Omega;\Mspace)$.
Therefore, by the Banach--Alaoglu theorem, there exists $\nu\in
L^\infty_w(\Omega;\Mspace)$ such that, for a (not relabeled) subsequence,
\begin{align}\label{eq:nun->nu}
  \nu^n\weak^\ast\nu\qquad\text{weak-$\ast$ in}~
  L^\infty_w(\Omega;\Mspace).
\end{align}

The next Lemma follows from the generalization of the fundamental
theorem on Young measures from \cite{Gwiazda:05} (see
Theorem~\ref{t:young_meas}) and Chacon's biting lemma
(Lemma~\ref{l:biting}).  Its proof is postponed to
Section~\ref{ss:young_meas}.
\begin{lem}\label{l:young_meas}
  With the notations of this section, $\nu_x$ is a probability measure
  a.e. in~$\Omega$ and there exists a nonincreasing sequence of
  measurable subsets $E_k\subset\Omega$, with $\abs{E_k}\to0$, such
  that for all $k\in\N$ we have for a (not relabeled) subsequence that
   \begin{align}\label{eq:bn->b}
     b_n(x)\weak
     \int_{\Rdds}\big(\bss(x,\vec{\zeta})-\bsS^\ast(x,\vec{\zeta})\big)
     :(\bsd(x,\vec{\zeta})-\nablas\vecu(x)) \,\d
     \nu_{x}(\vec{\zeta})\asdefined b(x)
   \end{align}
   weakly in $L^1(\Omega\setminus E_k)$ as $n\to\infty$. Moreover, we
   have that
   \begin{align*}
     \bsS(x)&=\int_{\Rdds}\bss(x,\vec{\zeta})\,\d \nu_{x}(\vec{\zeta})
     \qquad\text{and}\qquad
     \nablas \vecu(x)=\int_{\Rdds}
     \bsd(x,\vec{\zeta})\,\d \nu_{x}(\vec{\zeta}).
   \end{align*}
\end{lem}

\iffalse
Before completing the proof of Theorem \ref{t:main} with the help of
the preceding lemma, we observe from~\ref{A2} and~\ref{A3} that, for
each $\vec{\delta}\in \Rdd$ and a.e. $x\in \Omega$, the set
\begin{align}\label{eq:Cdelta}
  \mcC^x_{\vec{\delta}}\definedas\left\{\vec{\sigma}\in\Rdds\colon
    (\vec{\delta},\vec{\sigma})\in \Ax(x)\right\}
\end{align}
is convex.
We deduce from~\eqref{eq:SDv} and Lemma~\ref{l:young_meas}
that, to complete the proof, we need to show that
\begin{align}
  \label{eq:DsAx}
  \left(\nablas \vecu(x),\, \int_{\Rdds}\bss(x,\vec{\zeta})\,\d
    \nu_x(\vec{\zeta})\right)\in\Ax(x)\qquad\text{for a.e. $x \in \Omega$.}
\end{align}
To this end we shall collect some of the results established above.
Thanks to~\eqref{eq:AxChi} and~\eqref{df:s&d} we have,
  for a.e.~$x\in \Omega$, that
  \begin{align}\label{eq:dsAx}
    \big(\bsd(x,\vec{\zeta}),\,\bss(x,\vec{\zeta})\big)
    \in\Ax(x)\qquad\text{for all}~\vec{\zeta}\in\Rdds
  \end{align}
  and thanks to the properties of $\bsS^\ast$ stated in
  Section~\ref{ss:approxAx}, we have that
  \begin{align}\label{eq:DSAx}
    \big(\nablas \vecu(x),\bsS^\ast(x,\nablas \vecu)\big)\in
    \Ax(x)\qquad\text{for a.e. $x\in\Omega$.}
  \end{align}
  Combining Lemmas~\ref{l:bn->0} and~\ref{l:young_meas} we deduce that
  $b=0$ a.e. in $\Omega$. From this observation and~\ref{A2} it
  follows that
  \begin{align*}
    \left\{\vec{\zeta}\in\Rdds\colon
      \left(\bss(x,\vec{\zeta})-\bsS^\ast(x,\nablas\vecu(x)\right):
      \left(\bsd(x,\vec{\zeta})-\nablas\vecu(x)
      \right)>0\right\}\nsubset \supp\nu_x.
  \end{align*}
  Here, as usual, the support of the probability measure $\nu_x$ is
  defined to be the largest closed subset of $\Rdds$ for which every
  open neighborhood of every point of the set has positive measure.
  We split $\supp \nu_x$ into the following two sets:
  \begin{align*}
    \omega_1(x)&\definedas\left\{\vec{\zeta}\in\supp \nu_x\colon
      \bss(x,\vec{\zeta})=\bsS^\ast(x,\nablas\vecu(x))\right\},
    \\
    \omega_2(x)&\definedas \supp \nu_x\,\setminus \,\omega_1(x).
  \end{align*}
  As a consequence of~\ref{A2} we have that
  $\bsd(x,\vec{\zeta})=\nablas\vecu(x)$ for all $\vec{\zeta}\in
  \omega_2(x)$. Thus, for
  a.e. $x\in\Omega$ we deduce that
  \begin{multline*}
    \left(\int_{\omega_2(x)}\bsd(x,\vec{\zeta})\,\d \nu_x(\vec{\zeta})),
      \,\int_{\omega_2(x)}\bss(x,\vec{\zeta})\,\d \nu_x(\vec{\zeta})\right)
    \\
    =\nu_x\big(\omega_2(x)\big)
    \left(\nablas\vecu(x),\,\int_{\omega_2(x)}\bss(x,\vec{\zeta})
      \,\d \big(\tfrac{\nu_x(\vec{\zeta})}{\nu_x(\omega_2(x))}\big)\right).
  \end{multline*}
  The measure $\tfrac{\nu_x(\vec{\zeta})}{\nu_x(\omega_2(x))}$ is
  a probability measure on $\omega_2(x)$ and thus by~\eqref{eq:dsAx} it
  follows that $\int_{\omega_2(x)}\bss(x,\vec{\zeta})
  \,\d \left(\tfrac{\nu_x(\vec{\zeta})}{\nu_x(\omega_2(x))}\right)$ is a
  convex combination in the set $\mcC^x_{\nablas \vecu(x)}$; compare
  with~\eqref{eq:Cdelta}. Thus
  \begin{align}
    \label{eq:1}
    \left(\nablas\vecu(x),\,\int_{\omega_2(x)}\bss(x,\vec{\zeta}) \,\d
      \big(\tfrac{\nu_x(\vec{\zeta})}{\nu_x(\omega_2(x))}\big)\right)\in\Ax(x)
  \end{align}
  for a.e. $x\in\Omega$.  In view of Lemma~\ref{l:young_meas} we know
  that $\nu_x$ is a probability measure, \ie
  $\nu_x(\omega_1(x))+\nu_x(\omega_2(x))=1$.  Therefore, it follows from the
  identity for the shear rate in Lemma~\ref{l:young_meas}, that
  \begin{align*}
    \int_{\omega_1(x)}\bsd(x,\vec{\zeta})\,\d \nu_x(\vec{\zeta})&=
    \int_{\Rdds}\bsd(x,\vec{\zeta})\,\d
    \nu_x(\vec{\zeta})-\int_{\omega_2(x)}\bsd(x,\vec{\zeta})\,\d
    \nu_x(\vec{\zeta})
    \\
    &=\nablas \vecu(x)\big(1-\nu_x\big(\omega_2(x)\big)\big)=
    \nu_x\big(\omega_1(x)\big)\nablas \vecu(x).
  \end{align*}
  Hence, employing the definition of $\omega_1(x)$ we arrive at
  \begin{multline*}
    \left(\int_{\omega_1(x)}\bsd(x,\vec{\zeta})\,\d \nu_x(\vec{\zeta}),
      \,\int_{\omega_1(x)}\bss(x,\vec{\zeta})\,\d \nu_x(\vec{\zeta})\right)
    \\
    = \nu_x\big(\omega_1(x)\big) \left(\nablas\vecu(x)
      ,\,\bsS^\ast(x,\nablas\vecu(x))\right).
  \end{multline*}
  Thanks to~\eqref{eq:DSAx} we have that $\left(\nablas\vecu(x)
    ,\,\bsS^\ast(x,\nablas\vecu(x))\right)\in\Ax(x)$.
  Using once again the fact that $\nu_x$ is a probability measure we
  then deduce on noting~\eqref{eq:1} that
  \begin{multline*}
    \int_{\Rdds}\bss(x,\vec{\zeta})\,\d \nu_x(\vec{\zeta})=
    \int_{\omega_1(x)}\bss(x,\vec{\zeta})\,\d \nu_x(\vec{\zeta})
    +\int_{\omega_2(x)}\bss(x,\vec{\zeta})\,\d \nu_x(\vec{\zeta})
    \\
    =\nu_x(\omega_1(x)) \bsS^\ast(x,\nablas\vecu(x))+\nu_x(\omega_2(x))\int_{\omega_2(x)}\bss(x,\vec{\zeta})
    \,\d \big(\tfrac{\nu_x(\vec{\zeta})}{\nu_x(\omega_2(x))}\big)
  \end{multline*}
  is a convex combination of functions in $\mcC^x_{\nablas\vecu(x)}$.
  This proves~\eqref{eq:DsAx} and completes the proof of Theorem \ref{t:main}.
  \bigskip

\else
We deduce from~\eqref{eq:SDv} and Lemma~\ref{l:young_meas}
that, to complete the proof of Theorem~\ref{t:main}, we need to show that
\begin{align}
  \label{eq:DsAx}
  \left(\nablas \vecu(x),\, \int_{\Rdds}\bss(x,\vec{\zeta})\,\d
    \nu_x(\vec{\zeta})\right)\in\Ax(x)\qquad\text{for a.e. $x \in \Omega$.}
\end{align}
This follows from  the two preceding Lemmas exactly as in
\cite[p.\,131ff]{BulGwiMalSwi:09}. To be more precise, the proof is based on noting that for
each $\vec{\delta}\in \Rdd$ the set
\begin{align}\label{eq:Cdelta}
  \mcC^x_{\vec{\delta}}\definedas\left\{\vec{\sigma}\in\Rdds\colon
    (\vec{\delta},\vec{\sigma})\in \Ax(x)\right\}\quad\text{is convex for a.e. $x\in \Omega$;}
\end{align}
compare with~\ref{A2} and~\ref{A3}.
Combining Lemmas~\ref{l:bn->0} and~\ref{l:young_meas} we deduce that
$b=0$ a.e. in $\Omega$. Hence it follows from~\ref{A2} that
\begin{align*}
  \left\{\vec{\zeta}\in\Rdds\colon
    \left(\bss(x,\vec{\zeta})-\bsS^\ast(x,\nablas\vecu(x)\right):
    \left(\bsd(x,\vec{\zeta})-\nablas\vecu(x)
    \right)>0\right\}\nsubset \supp\nu_x.
\end{align*}
We split $\supp \nu_x$ into the two sets
\begin{align*}
  \omega_1(x)&\definedas\left\{\vec{\zeta}\in\supp \nu_x\colon
    \bss(x,\vec{\zeta})=\bsS^\ast(x,\nablas\vecu(x))\right\}\quad\text{and}\quad
  \omega_2(x)\definedas \supp \nu_x\,\setminus \,\omega_1(x).
\end{align*}
We investigate the pairing in
\eqref{eq:DsAx} on the two sets $\omega_1(x)$ and $\omega_2(x)$
separately. On $\omega_2(x)$ we have by~\ref{A2} that
$\bsd(x,\vec{\zeta})=\nablas\vecu(x)$. Therefore, on noting that
$\tfrac{\nu_x(\vec{\zeta})}{\nu_x(\omega_2(x))}$ is a probability
measure on $\omega_2(x)$, one can show
with~\eqref{eq:AxChi}, \eqref{df:s&d} and \eqref{eq:Cdelta} that
\begin{align*}
  \left(\nablas\vecu(x),\,\int_{\omega_2(x)}\bss(x,\vec{\zeta}) \,\d
    \big(\tfrac{\nu_x(\vec{\zeta})}{\nu_x(\omega_2(x))}\big)\right)\in\Ax(x)
  \quad\text{  for a.e. $x\in\Omega$.}
\end{align*}
On the other hand it follows from the definition of $\omega_1(x)$ that
\begin{align*}
  \int_{\omega_1(x)}\bss(x,\vec{\zeta})\,\d \nu_x(\vec{\zeta})
  = \nu_x\big(\omega_1(x)\big)
  \bsS^\ast(x,\nablas\vecu(x)).
\end{align*}
Thanks to the properties of $\bsS^\ast$, we have that $\left(\nablas\vecu(x)
  ,\,\bsS^\ast(x,\nablas\vecu(x))\right)\in\Ax(x)$; compare with
Section~\ref{ss:approxAx}.
Using the fact that $\nu_x$ is a probability measure, we
deduce  that
\begin{multline*}
  \int_{\Rdds}\bss(x,\vec{\zeta})\,\d \nu_x(\vec{\zeta})=
    \int_{\omega_1(x)}\bss(x,\vec{\zeta})\,\d \nu_x(\vec{\zeta})
    +\int_{\omega_2(x)}\bss(x,\vec{\zeta})\,\d \nu_x(\vec{\zeta})
    \\
    =\nu_x(\omega_1(x)) \bsS^\ast(x,\nablas\vecu(x))+\nu_x(\omega_2(x))\int_{\omega_2(x)}\bss(x,\vec{\zeta})
    \,\d \big(\tfrac{\nu_x(\vec{\zeta})}{\nu_x(\omega_2(x))}\big)
\end{multline*}
is a convex combination of functions. Moreover, due to the above
observations, for a.e. $x\in\Omega$, each of the two
functions in this convex combination is an element of the set
$\mcC^x_{\nablas\vecu(x)}$.
Hence, by \eqref{eq:Cdelta}, this completes the proof of Theorem \ref{t:main}.
\bigskip
\fi

  As in \cite{BulGwiMalSwi:09} we can establish from the above
  observations strong convergence of the symmetric velocity gradient
  and the stress on the subsets
  \begin{align*}
    \Omega_{\bsD}&\definedas \big\{x\in\Omega\colon
    \forall(\vec{\sigma}_1,\vec{\delta}_1)\in\Ax(x)~\text{with}~
    \big(\vec{\sigma}_1-\bsS^\ast(x,\nablas\vecu(x))\big):
    (\vec{\delta}_1-\nablas\vecu(x))=0
    \\
    &\qquad~\text{implies that}~\vec{\delta}_1=\nablas\vecu(x)\big\},\quad \mbox{and}
    \\
    \Omega_{\bsS}&\definedas \big\{x\in\Omega\colon
    \forall(\vec{\sigma}_1,\vec{\delta}_1)\in\Ax(x)~\text{with}~
    \big(\vec{\sigma}_1-\bsS^\ast(x,\nablas\vecu(x))\big):
    (\vec{\delta}_1-\nablas\vecu(x))=0
    \\
    &\qquad~\text{implies that}~\vec{\sigma}_1=\bsS^\ast(x,\nablas\vecu(x))\big\},
  \end{align*}
  respectively. Since the proof is identical to the proof of
  \cite[Lemma~5.2]{BulGwiMalSwi:09} we omit it here and we only state the
  result.

  \begin{cor}
    Assume the conditions of Theorem \ref{t:main} and let
    $r'\in(1,\infty)$ be such that $\frac1r+\frac1{r'}=1$. Then, for all
    $1\le s< r$ and $1\le s'<r'$, we have that, as $n \rightarrow \infty$,
    \begin{alignat*}{2}
      \nablas \Un&\to \new{\nablas} \vecu &\qquad&\text{strongly in}~
      L^{s}(\Omega_{\bsD})^{d\times d},
      \\
      \bSn&\to \bsS^\ast(\cdot,\nablas\vecu)&\qquad&\text{strongly in}~
      L^{s'}(\Omega_{\bsS})^{d\times d}.
    \end{alignat*}
  \end{cor}

\subsection{Proof of Lemma~\ref{l:bn->0}}
\label{ss:bn->0}

The proof of this assertion is motivated by the proof of \cite[Lemma
4.6]{BulGwiMalSwi:09}. However, since we are approximating problem
\eqref{eq:implicit} with finite element functions here, we need to use
the discrete Lipschitz truncation from~\sectionref{ss:Lipschitz}.

Let us define the auxiliary function
\begin{align}\label{df:an}
  a_n(x)\definedas \left(\bSn(x,\nablas\Un(x))-\bsS^\ast(x,\nablas
    \vecu(x))\right): (\nablas \Un(x)-\nablas \vecu(x))
\end{align}
and observe that
\begin{align}\label{eq:an-bn}
  \int_\Omega\abs{a_n-b_n}\,\d x \to 0\qquad\text{as}\quad n\to\infty.
\end{align}
Indeed, thanks to \eqref{eq:S^n} and the properties of $\eta^n$, we
have  that
\begin{align*}
  \int_\Omega\abs{a_n-b_n}\,\d x&=\int_\Omega\Big| \int_\Rdds \left(
    \bsS^\ast(x,\vec{\zeta})-\bsS^\ast(x,\nablas \vecu)\right): (\nablas
  \Un-\nablas \vecu) \,\d \mu_x^n(\vec{\zeta})
  \\
  &\qquad-\int_\Rdds \left( \bsS^\ast(x,\vec{\zeta})-\bsS^\ast(x,\nablas
    \vecu)\right): (\vec{\zeta}-\nablas \vecu) \,\d
  \mu_x^n(\vec{\zeta})\Big|\,\d x
  \\
  &=\int_\Omega\Big| \int_\Rdds \left(
    \bsS^\ast(x,\vec{\zeta})-\bsS^\ast(x,\nablas \vecu)\right): (\nablas
  \Un-\vec{\zeta}) \,\d \mu_x^n(\vec{\zeta})\Big|\,\d x
  \\
  &\le \int_\Omega \int_\Rdds \big|
  \bsS^\ast(x,\vec{\zeta})-\bsS^\ast(x,\nablas \vecu)\big|\,\big| \nablas
  \Un-\vec{\zeta}\big| \,\d \mu_x^n(\vec{\zeta})\,\d x
  \\
  &\leq \frac{c}{n}\int_\Omega
  \sup_{\abs{\vec{\zeta}-\nablas\Un(x)}\le\frac1n}
  \abs{\bsS^\ast(x,\vec{\zeta})-\bsS^\ast(x,\nablas\Un)}\,\d x\le \frac{c}n.
  \end{align*}
  Consequently, in order to prove that $b_n\to0$ in measure it
  suffices to prove that $a_n\to 0$ in measure. We shall establish the
  second claim in several steps.

  \textbf{Step 1:} First we introduce some preliminary facts
  concerning discrete Lipschitz truncations. For
  convenience we use the notation
  \begin{align*}
    \vecE^n \definedas \Pndiv(\Un-\vecu)=\Un-\Pndiv\vecu\in\Vn
  \end{align*}
  and let $\{\vecE^{n,j}\}_{n,j\in\N}\subset\Vn$ be the sequence of
  Lipschitz-truncated  finite element functions
  according to
  Corollary \ref{cor:dremlip}.  Recall from Lemma \ref{l:basic_conv}
  that $\vecE^n \weak 0$ weakly in $W^{1,r}_0(\Omega)^d$, \ie we are
  exactly in the situation of Corollary~\ref{cor:dremlip}.
  Although $\vecE^n\in\Vndiv$, \ie it is discretely divergence-free,
  this does not necessarily imply that $\vecE^{n,j}\in\Vndiv$
  and thus we need to modify $\vecE^{n,j}$ in order to be able
  to use it as a test function in \eqref{eq:Vn0}. Recalling Corollary
  \ref{c:Dbogovskii} we define
  \begin{subequations}\label{df:PhiPsi}
    \begin{align}
      \vec{\Psi}^{n,j}\definedas \Bogn(\divo \vecE^{n,j})\in \Vn.
    \end{align}
    The `corrected' function
    \begin{align}
      \vec{\Phi}^{n,j}\definedas \vecE^{n,j}-\vec{\Psi}^{n,j}\in\Vndiv
    \end{align}
  \end{subequations}
  is then discretely divergence-free. We need to control the correction in
  a norm.  To this end we recall from Section \ref{ss:fem_spaces} that
  $\Qn=\operatorname{span}\{Q^n_1,\ldots,Q^n_{N_n}\}$ for a certain locally
  supported basis. Then, thanks to properties of the discrete Bogovski{\u\i}
  operator and Corollary~\ref{c:Dbogovskii}, we have that
    \begin{align*}
      \beta_r\norm[1,r]{\vec{\Psi}^{n,j}}&\leq
       \sup_{Q\in\Qn}\frac{\scp{\divo \vecE^{n,j}}{Q}}{\norm[r']{Q}}
     =\sup_{Q\in\Qn}\frac{\scp{\divo \vecE^{n,j}-\divo
          \vecE^n}{Q}}{\norm[r']{Q}}\\
      &=\sup_{Q=\sum_{i=1}^{N_n}\rho_iQ_i^n}\Bigg( \sum_{\supp
        Q_i^n\subset \{\vecE^{n,j}=\vecE^{n}\}} \frac{\scp{\divo
          \vecE^{n,j}-\divo \vecE^n}{\rho_iQ_i^n}}{\norm[r']{Q}}\\
      &\hspace{2.5cm}+ \sum_{\supp
        Q_i^n\cap\{\vecE^{n,j}\neq\vecE^{n}\}\neq\emptyset}
      \frac{\scp{\divo \vecE^{n,j}-\divo
          \vecE^n}{\rho_iQ_i^n}}{\norm[r']{Q}}\Bigg)
      \\
      &=\sup_{Q=\sum_{i=1}^{N_n}\rho_iQ_i^n}\Bigg( \sum_{\supp
        Q_i^n\cap\{\vecE^{n,j}\neq\vecE^{n}\}\neq\emptyset}
      \frac{\scp{\divo \vecE^{n,j}-\divo
          \vecE^n}{\rho_iQ_i^n}}{\norm[r']{Q}}\Bigg)
      \\
      &=\sup_{Q=\sum_{i=1}^{N_n}\rho_iQ_i^n}\Bigg( \sum_{\supp
        Q_i^n\cap\{\vecE^{n,j}\neq\vecE^{n}\}\neq\emptyset}
      \frac{\scp{\divo \vecE^{n,j}}{\rho_iQ_i^n}}{\norm[r']{Q}}\Bigg)
      \\
      &\le \norm[r]{\divo
        \vecE^{n,j}\chi_{\Omega^n_{\{\vecE^{n,j}\neq\vecE^n\}}}}\sup_{Q=\sum_{i=1}^{N_n}\rho_iQ_i^n}\frac{\norm[r']{\sum_{\supp
        Q_i^n\cap\{\vecE^{n,j}\neq\vecE^{n}\}\neq\emptyset}\rho_iQ_i^n}}{\norm[r']{Q}}
      \\
      &
        \leq  c\,\norm[r]{\divo
        \vecE^{n,j}\,\chi_{\Omega^n_{\{\vecE^{n,j}\neq\vecE^n\}}}}\le
      c\,\norm[r]{\nabla
        \vecE^{n,j}\,\chi_{\Omega^n_{\{\vecE^{n,j}\neq\vecE^n\}}}},
  \end{align*}
  where $\chi_{\Omega^n_{\{\vecE^{n,j}\neq\vecE^n\}}}$ is the characteristic function of the
  set
  \begin{align*}
    \Omega^n_{\{\vecE^{n,j}\neq\vecE^n\}}\definedas\bigcup\left\{\Omega^n_\elm\mid
      \elm\in\gridn~\text{such that}~\elm\subset
      \overline{\{\vecE^{n,j}\neq\vecE^n\}} \right\}.
  \end{align*}
  Note that in the penultimate step of the above estimate we have used
  norm equivalence
  on the reference space $\hat\P_\Q$ from \eqref{df:Qn}. In
  particular, we see by means of standard scaling arguments that for
  $Q=\sum_{i=1}^{N_n}\rho_iQ_i^n$  the norms
  \begin{align*}
    Q \mapsto \Big(\sum_{i=1}^{N_n}\abs{\rho_i}^{r'}\norm[r']{Q_i^n}^{r'}\Big)^{1/r'}
    \qquad\text{and}
    \qquad Q \mapsto \norm[r']{Q}
  \end{align*}
  are equivalent with constants depending on the shape-regularity of
  $\mathbb{G}$ and $\hat\P_\Q$. This directly implies the desired estimate.

  Observe that
  $\abs{\Omega_\elm^n}\le c\abs{\elm}$ for all  $\elm\in\gridn$,
  $n\in\N$, with a shape-dependent constant $c>0$; hence,
  $\abs{\Omega^n_{\{\vecE^{n,j}\neq\vecE^n\}}}\le
  c\abs{{\{\vecE^{n,j}\neq\vecE^n\}}}$, and
  it follows from
  Theorem~\ref{thm:dTlest} and Corollary~\ref{cor:dremlip}\ref{itm:dremlip5} that
   \begin{align}
     \label{est:Psi}
     \beta_r\norm[1,r]{\vec{\Psi}^{n,j}}
     &\leq c\,\norm[r]{\lambda_{n,j}\chi_{\Omega^n_{\{\vecE^{n,j}\neq\vecE^n\}}}
     } \leq c\, 2^{-j/r} \norm[r]{\nabla \vec{E}^n}.
   \end{align}
   \new{Moreover, we have from Corollaries
  \ref{cor:dremlip} and~\ref{c:Dbogovskii}, that}
  \begin{subequations}\label{eq:PhiPsi}
    \begin{alignat}{3}
      \vec{\Phi}^{n,j},\vec{\Psi}^{n,j}&\weak 0&\qquad& \text{weakly
        in } W^{1,s}_0(\Omega)^d&\quad&\text{for
        all}~s\in[1,\infty),\label{eq:PhiPsia}
      \\
      \vec{\Phi}^{n,j},\vec{\Psi}^{n,j}&\to 0&\qquad &\text{strongly
        in } L^s(\Omega)^d&\quad&\text{for
        all}~s\in[1,\infty),\label{eq:PhiPsib}
    \end{alignat}
  \end{subequations}
  as $n\to\infty$.

  \textbf{Step 2:} We claim that
  \begin{align*}
    \underset{n\to\infty}{\lim\sup}
    \int_{\{\vecE^n=\vecE^{n,j}\}}\abs{a_n}\,\d x\leq c\,2^{-j/r},
  \end{align*}
  with a constant $c>0$ independent of $j$.
  To see this we first observe that $\abs{a_n}=a_n+2a_n^-$
  with the usual notation $a_n^-(x)=\max\{-a_n(x),0\}$, $x\in\Omega$.
  Therefore, we have that
  \begin{align}\label{eq:an=}
    \begin{split}
      \underset{n\to\infty}{\lim\sup}
      \int_{\{\vecE^n=\vecE^{n,j}\}}\abs{a_n}\,\d x&\le\underset{n\to\infty}{\lim\sup}
      \int_{\{\vecE^n=\vecE^{n,j}\}}a_n\,\d x
      \\
      &\qquad+2\,\underset{n\to\infty}{\lim\sup}
      \int_{\{\vecE^n=\vecE^{n,j}\}}a_n^-\,\d x.
    \end{split}
  \end{align}
  We bound the two terms on the right-hand side separately.
  As a consequence of \eqref{eq:an-bn}
  and the fact that $b_n(x)\geq0$ for a.e. $x\in\Omega$ (cf.
  \eqref{df:bn}) it follows that
  \begin{align*}
    \int_{\{\vecE^n=\vecE^{n,j}\}}a_n^-\,\d x\le
    \int_\Omega a_n^-\,\d x \leq \int_\Omega \abs{a_n-b_n}\,\d x
    \to 0, \qquad\text{as } n\to\infty.
  \end{align*}
  The bound on the first term on the right-hand side of \eqref{eq:an=} is more involved. In particular,
  recalling the definitions \eqref{df:an} and \eqref{df:PhiPsi} we
  have that
  \begin{align*}
    &\int_{\{\vecE^n=\vecE^{n,j}\}}a_n\,\d x
    \\
    &\qquad=\int_{\{\vecE^n=\vecE^{n,j}\}}
    \left(\bSn(\cdot,\nablas\Un)-\bsS^\ast(\cdot,\nablas
      \vecu)\right): (\nablas \Pndiv\vecu-\nablas \vecu)\,\dx
    \\
    &\qquad\qquad+\int_\Omega
    \bSn(\cdot,\nablas \Un): \nablas \vec{\Phi}^{n,j}\,\d x +
    \int_\Omega \bSn(\cdot,\nablas \Un): \nablas
    \vec{\Psi}^{n,j}\,\d x
    \\
    &\qquad\qquad -\int_\Omega \bsS^\ast(\cdot,\nablas \vecu): \nablas
    \vecE^{n,j}\,\d x
    \\
    &\qquad\qquad+ \int_{\{\vecE^n\neq\vecE^{n,j}\}}
    \big(\bsS^\ast(\cdot,\nablas\vecu)-\bSn(\cdot,\nablas \Un)\big):
    \nablas \vecE^{n,j}\,\d x
    \\
    &\qquad= \text{I}^{n}+
    \text{II}^{n,j}+\text{III}^{n,j}+\text{IV}^{n,j}+\text{V}^{n,j}.
  \end{align*}
  Thanks to \eqref{eq:approxVn} and \eqref{eq:bound} we have that, as $n \rightarrow \infty$,
  \begin{align*}
    \abs{\text{I}^{n}}&\le \int_{\{\vecE^n=\vecE^{n,j}\}}
    \abs{\bSn(x,\nablas\Un(x))-\bsS^\ast(x,\nablas
      \vecu(x))}\abs{\nablas \Pndiv\vecu(x)-\nablas \vecu(x)}\,\dx
    \\
    &\le \norm[r']{\bSn(\cdot,\nablas\Un(\cdot))-\bsS^\ast(\cdot,\nablas
      \vecu(\cdot)}\norm[r]{\nablas \Pndiv\vecu-\nablas \vecu}\to 0.
  \end{align*}
  In order to estimate $\text{II}^{n,j}$ we recall that
  $\vec{\Phi}^{n,j}\in\Vndiv$ is discretely divergence-free, and
  we can therefore use it as a test function in \eqref{eq:Vn0} to deduce that
  \begin{align*}
    \text{II}^{n,j}=-\Trilin{\Un}{\Un}{\vec{\Phi}^{n,j}}
    +\scp{\vecf}{\vec{\Phi}^{n,j}}\to 0\qquad\text{as}~ n\to\infty.
  \end{align*}
  Indeed, the second term vanishes thanks to \eqref{eq:PhiPsia}. The
  first term vanishes by arguing as in \eqref{eq:Bn->B} --- observe
  that for \eqref{eq:B1} the weak convergence \eqref{eq:PhiPsia} of
  $\vec{\Phi}^{n,j}$ is sufficient.
  The term $\text{III}^{n,j}$ can be
  bounded by means of \eqref{est:Psi}; in particular,
   \begin{align*}
     \underset {n\to\infty} {\lim\,\sup}~\abs{\text{III}^{n,j}} \leq
     \underset {n\to\infty} {\lim\,\sup}\,
     \norm[r']{\bsS(\cdot,\nablas
       \Un)}\norm[r]{\nablas\vec{\Psi}^{n,j}}\le c\, 2^{-j/r},
   \end{align*}
   where we have used \eqref{eq:bound}.
   Corollary \ref{cor:dremlip} implies that
   \begin{align*}
     \lim_{n\to\infty}\text{IV}^{n,j}=0.
   \end{align*}
   Finally, by \eqref{eq:bound} and Corollary \ref{cor:dremlip}, we have that
   \begin{align*}
     \underset {n\to\infty} {\lim\,\sup}~\abs{\text{V}^{n,j}}&\leq
     \underset {n\to\infty}
     {\lim\,\sup}\,\big(\norm[r']{\bsS^\ast(\cdot,\nablas
       \vecu)}+\norm[r']{\bSn(\cdot,\nablas\Un)}
     \big)\norm[r]{\nablas
       \vecE^{n,j}\chi_{\{\vecE^n\neq\vecE^{n,j}\}}}
     \\
     &\leq c\,2^{-j/r}.
   \end{align*}
   In view of \eqref{eq:an=} this completes Step 2.

  \textbf{Step 3:}   
  We prove, for any $\vartheta\in(0,1)$, that
  \begin{align*}
    \lim_{n\to\infty}\int_\Omega\abs{a_n}^\vartheta\,\d x=0.
  \end{align*}
  Using H{\"o}lder's inequality, we easily obtain that
  \begin{align*}
    \int_\Omega\abs{a_n}^\vartheta\,\d
    x&=\int_{\{\vecE^n=\vecE^{n,j}\}}\abs{a_n}^\vartheta\,\d x
    +\int_{\{\vecE^n\neq\vecE^{n,j}\}}\abs{a_n}^\vartheta\,\d x
    \\
    &\leq
    \abs{\Omega}^{1-\vartheta}
    \left(\int_{\{\vecE^n=\vecE^{n,j}\}}\abs{a_n}\,\d x\right)^{\vartheta}
    + \left(\int_\Omega\abs{a_n}\,\d x\right)^{\vartheta}
    \abs{\{\vecE^n\neq\vecE^{n,j}\}}^{1-\vartheta}.
  \end{align*}
  Thanks to \eqref{eq:bound}, we have that
  $(\int_\Omega\abs{a_n}\,\d x)^{\vartheta}$ is bounded uniformly in $n$
  and by Corollary \ref{cor:dremlip} we have that
  \begin{align*}
    \abs{\{\vecE^n\neq\vecE^{n,j}\}}\leq c\,
    \frac{\norm[1,r]{\vecE^n}^r}{\lambda_{n,j}^r}
    \leq \frac{c}{2^{2^jr}},
  \end{align*}
  where we have used that $\{\vecE^n\}_{n\in\N}$ is
  bounded in $W^{1,r}_0(\Omega)^d$ according to \eqref{eq:bound} and
  Assumption \ref{ass:Pndiv}. Consequently, from Step 2 we deduce that
  \begin{align*}
    \underset {n\to\infty}
    {\lim\,\sup}\int_\Omega\abs{a_n}^\vartheta\,\d x&\le
    c\abs{\Omega}^{1-\vartheta} 2^{-j\vartheta/r} +
    \frac{c}{2^{2^jr(1-\vartheta)}}.
  \end{align*}
  The left-hand side is independent of $j$ and we can thus pass to the
  limit $j\to\infty$. This proves the assertion and
  actually implies that $a_n\to 0$ in measure as
  $n\to\infty$. According to \eqref{eq:an-bn} we have that $b_n\to0$ in
  measure and thus we have completed the proof of Lemma \ref{l:bn->0}.

\subsection{Proof of Lemma~\ref{l:young_meas}\label{ss:young_meas}}
The proof of this Lemma is given in~\cite{BulGwiMalSwi:09}. In order to keep the paper
self-contained, we shall reproduce it here.

The assertion is an immediate consequence of the result on
Young measures from~\cite{Gwiazda:05} stated in
Theorem~\ref{t:young_meas}. It therefore
suffices to check the  assumptions therein. The first assumption has already been
verified in~\eqref{eq:nun->nu}.

\textbf{Step 1:} We prove that the sequence $\{\nu^n\}_{n\in\N}$
  satisfies the tightness condition. From the definition of
  $\nu_x^n$ (cf. \eqref{df:nun})  it follows  that
  \begin{align*}
    \gamma^n(x)\definedas
    \max_{\vec{\zeta}\in\supp\nu_x^n}\abs{\vec{\zeta}}=
      \max_{\vec{\zeta}\in\supp\mu_x^n}\abs{\Gx(\vec{\zeta})} \leq
      \max_{\vec{\zeta}\in\supp\mu_x^n}
      \left(\abs{\vec{\zeta}} +\abs{\bsS^\ast(x,\vec{\zeta})}\right).
  \end{align*}
  We deduce from the inclusion $\supp \mu_x^n\subset B_{1/n}(\nablas \vecVn(x))$
  that $\norm[s]{\gamma^n}\le c$ for some constant~$c>0$ and
  $s=\max\{r,r'\}>1$.
  Since $\Omega$ is bounded, the sequence is uniformly bounded  in
  $L^1(\Omega)$,  and for $M>0$ we have
  \begin{align*}
    \abs{\big\{x\in\Omega\colon \supp\nu_x^n \setminus B_{M}(0)\big\}}
    &=\abs{\big\{x\in\Omega\colon \gamma^n(x)>M\big\}}
    \le \int_\Omega\frac{\gamma^n(x)}M\,\d x\le \frac{c}{M}.
  \end{align*}
  This yields the tightness of $\{\nu^n\}_{n\in\N}$ and it follows
  from part \ref{i} of Theorem \ref{t:young_meas} that $\nu_x$ is a
  probability measure, \ie $\norm[\Mspace]{\nu_x}=1$ for a.e. $x\in\Omega$.

\textbf{Step 2:}
We turn to proving \eqref{eq:bn->b}. Recalling \eqref{df:SUbc}, the
assertion follows if there exists a nonincreasing sequence of
measurable subsets $\{E_i\}_{i\in\N}$ with $\abs{E_i}\to0$ as
$i\to\infty$, such that the function
\begin{align*}
  h(x,\vec{\zeta})\definedas
  \left(\bss(x,\vec{\zeta})-\bsS^\ast(x,\nablas
    \vecu(x))\right):(\bsd(x,\vec{\zeta})-\nablas\vecu(x))
\end{align*}
satisfies \eqref{eq:carath} on $A=\Omega\setminus E_i$ for each $i\in\N$.
This can be seen as follows. From \eqref{eq:bound} it follows that
$\norm[r]{\nablas\Un}+\norm[r]{\nablas\vecu}$ is bounded uniformly in
$n\in\N$. Consequently, the sequence $\{c_n\}_{n\in\N}$, defined by
\begin{multline}\label{eq:cn}
  c_n(x)\definedas c\Big(\abs{\nablas \Un(x)}^{r-1}+\abs{\nablas
    \vecu(x)}^{r-1}+\frac1{n^{r-1}}+k^{\frac1{r'}}(x)\Big)
  \\\times\Big(\abs{\nablas \Un(x)}+\abs{\nablas
    \vecu(x)}+\frac1n\Big),
\end{multline}
is bounded in $L^1(\Omega)$, where $k\in L^{r'}(\Omega)$ from
\eqref{eq:S*coerc} and $c>0$ is a constant
 to be chosen later.
Hence, Chacon's biting lemma (Lemma
\ref{l:biting}) implies that there exists a nonincreasing sequence of
measurable subsets $\{E_i\}_{i\in\N}$ with $\abs{E_i}\to0$ as
$i\to\infty$ such that $\{c_n\}_{n\in\N}$ is weakly precompact in
$L^1(\Omega\setminus E_i)$ for each $i\in\N$. Fix $i\in\N$ and set
$\omega\definedas \Omega\setminus E_i$. Thanks to the
de la Valle{\'e}-Poussin theorem (see, \cite{Meyer:66}), there exists
a nonnegative increasing
convex function $\phi:\R_+\to\R_+$ such that
\begin{align}\label{eq:phi}
  \frac{\phi(t)}t\to \infty\quad\text{as}\quad t\to\infty\qquad\text{and}\qquad\sup_{n\in\N} \int_\omega\phi(\abs{c_n})\,\d x<\infty.
\end{align}
Let us also define
\begin{align*}
  H(x,\vec{\zeta})\definedas
  \left(\bsS^\ast(x,\vec{\zeta})-\bsS^\ast(x,\nablas\vecu(x))\right):
  (\vec{\zeta}-\nablas \vecu(x)).
\end{align*}
By a simple substitution in the spirit of \eqref{df:SUb} it follows that
\begin{align*}
  \sup_{n\in\N}&\int_\omega\int_{\left\{\vec{\zeta}\in\Rdds\colon
      \abs{h(x,\vec{\zeta})}>R\right\}}h(x,\vec{\zeta})\,\d \nu^n_x(\vec{\zeta})\,\d x
  \\
  &=\sup_{n\in\N}\int_\omega\int_{\left\{\vec{\xi}\in\Rdds\colon
      \abs{H(x,\vec{\xi})}>R\right\}}H(x,\vec{\xi})\,\d \mu^n_x(\vec{\xi})\,\d x
  \\
  &\le \sup_{t\ge R}\frac{t}{\phi(t)}\underbrace{\sup_{n\in\N}
  \int_\omega\int_{\left\{\vec{\xi}\in\Rdds\colon
      \abs{H(x,\vec{\xi})}>R\right\}}
  \phi\big(H(x,\vec{\xi})\big)\,\d \mu^n_x(\vec{\xi})\,\d x}_{\asdefined J_R}\!.
\end{align*}
Thanks to the properties \eqref{eq:phi} of $\phi$ the assertion follows once $J_R$
has been shown to remain bounded. To this end, we observe that
\begin{align*}
  J_R&\le \sup_{n\in\N}\int_\omega\, \sup_{\vec{\xi}\in
    B_{\frac1n}(\nablas\vecu(x))}\, \phi\big(H(x,\vec{\xi})\big)\,\d x
  \le  \sup_{n\in\N}\int_\omega\phi(\abs{c_n})\,\d x,
\end{align*}
where we have used that we can choose the constant in \eqref{eq:cn} so that
\begin{align*}
  H(x,\vec{\xi})\leq c \left(\abs{\vec{\xi}}^{r-1}+\abs{\nablas
      \vecu(x)}^{r-1}+k(x)\right)\left(\abs{\vec{\xi}} +\abs{\nablas\vecu(x)}\right).
\end{align*}
The assertion then follows from \eqref{eq:phi}.

Finally the identities for $\nablas\vecu$ and $\bsS$ follow similarly
from the representations \eqref{df:SUba} and \eqref{df:SUbb} and the
uniqueness of the weak limits \eqref{eq:weakUn} and
\eqref{eq:weakSn}.

Thus we have completed the proof of Lemma~\ref{l:young_meas}.


\section{Conclusions}\label{sec:conclusion}
  We have established the convergence of finite element approximations of implicitly constituted
  power-law-like models for viscous incompressible fluids. A key new technical tool
  in our analysis was a finite element counterpart of the Acerbi--Fusco Lipschitz
  truncation of Sobolev functions, which was used in combination with a variety of
  weak compactness techniques, including Chacon's biting lemma and Young measures.
  An interesting direction for future research is the extension of the results
  obtained herein to unsteady implicitly constituted models of incompressible
  fluids.

\section*{Acknowledgments}
The research of Christian Kreuzer was supported by the DFG research grant KR 3984/1-1.
Lars Diening and Endre S{\"u}li were supported by the EPSRC Science and Innovation award to the Oxford Centre
for Nonlinear PDE (EP/E035027/1).

\providecommand{\bysame}{\leavevmode\hbox to3em{\hrulefill}\thinspace}
\providecommand{\MR}{\relax\ifhmode\unskip\space\fi MR }
\providecommand{\MRhref}[2]{%
  \href{http://www.ams.org/mathscinet-getitem?mr=#1}{#2}
}
\providecommand{\href}[2]{#2}

\section*{Appendix: Auxiliary comments on Lipschitz truncation}
\label{sec:appendix}
Although
similar techniques were used in~\cite{BreDieFuc:12}
to prove the properties of the Lipschitz truncation,
we decided to present a complete
proof of Theorem~\ref{thm:Tlest} for the following two reasons:
\begin{itemize}[leftmargin=0.5cm]
\item In contrast with the Lipschitz truncation in~\cite{BreDieFuc:12},
  the Lipschitz truncation in~\eqref{eq:lip} preserves boundary
  values. This requires changes to the proof that are not always obvious.
\item The concept of Lipschitz truncation seems
  to be new to the numerical analysis community. For this reason we have aimed
  to keep the presentation as self-contained as possible.
\end{itemize}
Recall the notational conventions introduced in Section \ref{ss:Lipschitz} prior to Theorem \ref{thm:Tlest},
and the definition \eqref{eq:lip} of the Lipschitz truncation. We start with some basic estimates.

\begin{lem}
  \label{lem:QjPoin}Let $\lambda>0$ and $\vecv\in
  W^{1,1}_0(\Omega)^d$ and let $\{\vecv_j\}_{j\in\N}\subset\Rd$ be
  defined as in \eqref{eq:lip2}. We then have, for all $j\in\N$, that
  \begin{enumerate}[leftmargin=1cm,itemsep=1ex,label={\rm(\alph{*})}]
  \item
    \label{itm:QjPoin1}
    $\fint_{\mcQ_j^{\ast\ast}}\left|\frac{\vecv-\vecv_j}{\ell_j}\right|\,\dx
    \leq c\, \fint_{\mcQ_j^{\ast\ast}} |\nabla \vecv|\,\dx \leq c\,
    M(\nabla \vecv)(y)$ for all $y \in \mcQ_j^{\ast\ast}$;
  \item
    \label{itm:QjPoin1b}
   $\fint_{\mcQ_j^{\ast\ast}} \abs{\nabla \vecv} \,\dx \leq c\,\lambda$;
  \item
    \label{itm:QjPoin2}
    for $k \in \N$ with $\mcQ_j^\ast \cap \mcQ_k^\ast \neq
    \emptyset$, we have
    \begin{align*}
      \abs{\vecv_j - \vecv_k} \leq c
      \fint_{\mcQ_j^{\ast\ast}}\left|\vecv-\vecv_j\right|\,\dx + c
      \fint_{\mcQ_k^{\ast\ast}}\left|\vecv-\vecv_k\right|\,\dx;
    \end{align*}
  \item \label{itm:QjPoin3} for $k \in \N$ with $\mcQ_j^\ast \cap \mcQ_k^\ast
    \neq \emptyset$ we have
    $\abs{\vecv_j - \vecv_k} \leq c\, \ell_j\, \lambda$.
  \end{enumerate}
\end{lem}

\begin{proof}
  We extend $\vecv$ by zero outside $\Omega$.

  \ref{itm:QjPoin1} This statement
  is a consequence of \Poincare{}'s inequality and the Friedrichs inequality.
  Indeed, recalling \eqref{eq:lip2}, for $\mcQ_j^* \subset
  \Omega$ we have by \Poincare's inequality that
  \begin{align*}
    \fint_{\mcQ_j^{\ast\ast}}\left|\frac{\vecv-\vecv_j}{\ell_j}\right|\,\dx&\leq
    c\, \fint_{\mcQ_j^{\ast\ast}} \abs{\nabla \vecv}\,\dx \leq
    c\, \fint_{B_{\diam(\mcQ_j^{**})}(y)} \abs{\nabla \vecv}\,\dx\le
    c\, M(\nabla
    \vecv)(y)
  \end{align*}
  for all $y \in \mcQ_j^{\ast\ast}$; the constant $c$ depends only on $d$.

  In the case $\mcQ_j^* \nsubset \Omega$, it follows from the fact
  that $\Omega$ is a Lipschitz domain and
  $\mcQ_j^{**}=\sqrt{\frac98}\mcQ_j^*$, that $\abs{Q_j^{\ast\ast}
    \setminus \Omega} \geq c\abs{\mcQ_j^\ast}$, with a constant $c>0$
  depending on~$\Omega$.  Hence $\vecv$ is zero on a portion of
  $\mcQ^{\ast\ast}_j$ whose measure is bounded below by a positive constant,
  which depends on the Lipschitz constant of $\partial\Omega$.
  Consequently, we can apply Friedrichs' inequality (cf.~\cite[Lemma
  1.65]{MalyZiemer:97}) to deduce that
  \begin{align*}
    \fint_{\mcQ_j^{\ast\ast}}\left|\frac{\vecv}{\ell_j}\right|\,\dx&\leq c\,
    \fint_{\mcQ_j^{\ast\ast}} \abs{\nabla \vecv}\,\dx \leq c\, M(\nabla
    \vecv)(y)\quad \forall y \in \mcQ_j^{\ast\ast}.
  \end{align*}

  \ref{itm:QjPoin1b}
  It follows from \ref{itm:Wbnd} that
  $(\theta_d \mcQ_j) \cap (\R^d \setminus \mathcal{U}_\lambda(\vecv)) =(\theta_d \mcQ_j) \cap \{M(\nabla
  \vecv) \leq \lambda\}
  \neq \emptyset$; compare with \eqref{df:UH}. For $z \in
  (\theta_d \mcQ_j) \cap \{M(\nabla \vecv) \leq
  \lambda\}$ let $R_j :=\theta_d
  \sqrt{d}\frac98\ell_j=\theta_d\diam(\mcQ^{\ast\ast}_j)$; then,
  $\theta_d\mcQ_j^{\ast\ast}\subset B_{R_j}(z)$. Consequently,
  \begin{align*}
    \fint_{\mcQ_j^{\ast\ast}} \abs{\nabla \vecv} \,\dx &\leq c\, \fint_{\theta_d
      \mcQ_j^{\ast\ast}} \abs{\nabla \vecv} \,\dx \leq c\, \fint_{B_{R_j}(z)}
    \abs{\nabla \vecv} \,\dx \leq c\, M(\nabla \vecv)(z) \leq c\,
    \lambda.
  \end{align*}

  \ref{itm:QjPoin2} Observe that
  $\mcQ_j^\ast\cap\mcQ_k^\ast\neq\emptyset$ is equivalent to
  $\mcQ_j\cap\mcQ_k\neq\emptyset$ and hence we obtain from
  \ref{itm:W3} and
  $\mcQ_i^{\ast\ast}=\sqrt{\frac98}\mcQ_i^\ast$, $i\in\N$, that
  $$\abs{\mcQ_j^{\ast\ast} \cap \mcQ_k^{\ast\ast}} \geq (4\sqrt{2})^{-d}
  \max
  \{\abs{\mcQ_j^*}, \abs{\mcQ_k^*}\}.$$
  Therefore, there exists a constant $c>0$, depending on $d$, such that
  \begin{align*}
    \abs{\vecv_j - \vecv_k} &\leq \fint_{\mcQ_j^{\ast\ast} \cap
      \mcQ_j^{**}} \abs{\vecv - \vecv_j}\, \dx +
    \fint_{\mcQ_j^{\ast\ast} \cap \mcQ_j^{**}} \abs{\vecv - \vecv_k}\,
    \dx
    \\
    &\leq c \fint_{\mcQ_j^{\ast\ast}}\left|\vecv_j-\vecv\right|\,\dx + c
    \fint_{\mcQ_k^{\ast\ast}}\left|\vecv-\vecv_k\right|\,\dx.
  \end{align*}

  \ref{itm:QjPoin3} The claim is a combination of~\ref{itm:QjPoin2},
  \ref{itm:QjPoin1}, \ref{itm:QjPoin1b} and \ref{itm:W3}.
\end{proof}

The next result proves that the Lipschitz truncation is a
proper Sobolev function.
\begin{lem}
  \label{lem:TlW110}
  Let $\lambda>0$, $\vecv \in W^{1,1}_0(\Omega)^d$ and  let
  $\vecv_\lambda$ be defined as in \eqref{eq:lip}. Then,
   $\vecv_\lambda-\vecv = \sum_{j
    \in \N} \psi_j (\vecv_j -\vecv) \in
  W^{1,1}_0(\mcU_\lambda(\vecv) \cap \Omega)^d$.
\end{lem}

\begin{proof}
  It follows from \eqref{eq:lip} and properties of the partition of
  unity $\{\psi_j\}_{j\in\N}$ that $\vec{v}_\lambda - \vec{v} =
  \sum_{j \in \N} \psi_j (\vec{v}_j - \vec{v})$ pointwise on~$\R^d$
  and $\vecv_\lambda- \vecv=0$ in the complement
  of~$\mcU_\lambda(\vecv)$. Moreover, we have that $\psi_j (\vecv_j
  - \vecv) \in W^{1,1}_0(\mcU_\lambda(\vecv) \cap \Omega)^d$. Indeed,
  for $\mcQ_j^* \subset \Omega$ this follows from the fact that $\psi \in
  C^\infty_0(\mcQ_j^*)$. If on the other hand $\mcQ_j^* \not\subset
  \Omega$, then this follows from $\vecv_j=0$ and $\vecv \in
  W^{1,1}_0(\mcU_\lambda(\vecv) \cap \Omega)$. We need to show that the sum
  converges in $W^{1,1}_0(\mcU_\lambda(\vecv) \cap \Omega)^d$.  Since
  $\Omega$ is bounded, it suffices to prove that the sum of the
  gradients converges absolutely in~$L^1(\Omega)^{d \times d}$.
  We have, pointwise,  the equality
  \begin{equation*}
    \sum_{j \in \N} \nabla \Big( \psi_j(\vecv_j - \vecv) \Big)
    =\sum_{j \in \N} \big( (\nabla \psi_j)(\vecv_j - \vecv) + \psi_j
    (\nabla \vecv_j - \nabla \vecv) \big),
  \end{equation*}
  where we have used that both sums are just finite sums, since the family
  $\mcQ_j^\ast$ is locally finite. Every summand in the last sum
  belongs to $L^1(\mcU_\lambda(\vecv) \cap \Omega)^{d\times d}$. For a
  finite subset $I \subset \N$, we have, thanks to
  Lemma~\ref{lem:QjPoin} and the locally finite overlaps of the
  $\mcQ_j^\ast$, that
  \begin{align*}
    \Sigma_I &:= \int_{\mcU_\lambda(\vecv)} \sum_{j \in \N \setminus
      I} \abs{ (\nabla \psi_j)(\vecv_j - \vecv) + \psi_j (\nabla
      \vecv_j - \nabla \vecv) } \,\dx
    \\
    &\leq c\sum_{j \in \N \setminus I} \int_{\mcQ_j^\ast}
    \frac{\abs{\vecv_j - \vecv}}{\ell_j} \,\dx + \sum_{j \in \N
      \setminus I} \int_{\mcQ_j^\ast} \abs{\nabla \vecv} \,\dx
    \\
    &\leq c\sum_{j \in \N \setminus I} \int_{\mcQ_j^{\ast\ast}}
    \abs{\nabla \vecv}\,\dx\leq c\sum_{j \in \N \setminus I} \lambda\,
    \abs{\mcQ_j}
    \leq c\, \int_{\mcU_\lambda(\vecv)} \chi_{\cup_{j \in \N
        \setminus I} \mcQ_j^\ast} \lambda\,\dx.
  \end{align*}
  Note that $\bigcup_{j \in \N \setminus I}\mcQ_j^\ast \subset
  \mcU_\lambda(\vecv)$ and $\lambda\, \abs{\mcU_\lambda(\vecv)} =
  \lambda \abs{\{M(\nabla\vecv)> \lambda\}} \leq c\,
  \norm[L^1(\Omega)]{\nabla \vecv}$ by the weak type
  estimate~\eqref{eq:HLmax_weak} and $\vecv \in W^{1,1}_0(\Omega)^d$.
  Thus, $\chi_{\cup_{j \in \N \setminus I} \mcQ_j^\ast} \lambda \leq
  \chi_{\mcU_\lambda(\vecv)} \lambda \in L^1(\R^d)$.  Therefore, it follows
  by $\chi_{\bigcup_{j \in \N \setminus I}\mcQ_j^\ast} \to 0$ and the
  Lebesgue's dominated convergence theorem that $\Sigma_I\to 0$ as $I
  \to \N$. Hence the sum 
  $\sum_{j \in \N} \nabla \Big( \psi_j(\vecv_j - \vecv) \Big)$ converges absolutely in
  $L^1(\Omega)^{d\times d}$, and the claim follows.
\end{proof}

\bigskip

\noindent
\begin{proof}[Proof of Theorem \ref{thm:Tlest}] We shall consider parts
\ref{itm:Tlestsupp}--\ref{itm:TlestLinfty} in the statement of the theorem separately.

  \ref{itm:Tlestsupp} The claim directly follows from $\vecv_\lambda-
  \vecv \in W^{1,1}_0(\mcU_\lambda \cap \Omega)^d$ (see Lemma~\ref{lem:TlW110}).

  \ref{itm:TlestL1} We begin by noting that
  \begin{align*}
    \chi_{\mcU_\lambda(\vecv)} \abs{\vecv_\lambda} &\leq \sum_{j \in
      \N} \chi_{\mcQ_j^\ast} \abs{\vec{v}_j} \leq \sum_{j \in \N}
    \chi_{\mcQ_j^\ast} \fint_{\mcQ_j^{\ast\ast}} \abs{\vec{v}} \,\dx.
  \end{align*}
  By Jensen's inequality and the local finiteness of the
  $\mcQ_j^{\ast\ast}$ we then deduce that
  \begin{align*}
    \int_\Omega \chi_{\mcU_\lambda(\vecv)} \abs{\vecv_\lambda}^s\, \dx
    &\leq \sum_{j \in \N} \abs{\mcQ_j^\ast} \bigg(
    \fint_{\mcQ_j^{\ast\ast}} \abs{\vec{v}} \,\dx\bigg)^s \leq \sum_{j
      \in \N} \int_{\mcQ_j^{\ast\ast}} \abs{\vec{v}}^s \,\dx \leq c\,
    \int_\Omega \abs{\vec{v}}^s\,\dx,
  \end{align*}
  for $s \in [1,\infty)$, which then proves~\ref{itm:TlestL1} for $s
  \in [1,\infty)$ using also that $\vecv_\lambda = \vecv$ outside
  of~$\mcU_\lambda(\vecv)$. The case $s=\infty$ follows by obvious
  modifications of the argument.

  \ref{itm:TlestW1h}
  We define $I_j := \{k \in \N \,:\, \mcQ_j^\ast \cap \mcQ_k^\ast \neq
  \emptyset\}$. Then, on every $\mcQ_j^\ast$ we have that
  \begin{align*}
    \nabla \vecv_\lambda  &= \nabla \Big(\sum_{k \in \N}
    \psi_k \vecv_k \Big) = \nabla \Big(\sum_{k \in I_j}
    \psi_k (\vecv_k -\vecv_j) \Big)  = \sum_{k \in I_j}  (\nabla
    \psi_k) (\vecv_k - \vecv_j),
  \end{align*}
  where we have used that $\sum_{k \in I_j} \psi_k =1$ on $\mcQ_j^\ast$. By
  Lemma~\ref{lem:QjPoin} we thus obtain
  \begin{align*}
    \chi_{\mcU_\lambda(\vecv)} \abs{\nabla \vecv_\lambda} &\leq
    c\,\sum_{j \in \N} \chi_{\mcQ_j^\ast} \sum_{k \in I_j}
    \fint_{\mcQ_k^\ast} \frac{\abs{\vecv - \vecv_k}}{\ell_k} \,\dx \leq
    c\, \sum_{j \in \N} \chi_{\mcQ_j^\ast} \sum_{k \in I_j}
    \fint_{\mcQ_k^{\ast\ast}} \abs{\nabla \vecv} \,\dx.
  \end{align*}
  The inequality in part~\ref{itm:TlestW1h} now follows by arguing as
  in part~\ref{itm:TlestL1}.

  \ref{itm:TlestLinfty} It follows from the final chain of inequalities in
  the proof of part~\ref{itm:TlestW1h} above, 
  Lemma~\ref{lem:QjPoin}~\ref{itm:QjPoin1b} and the local finiteness
  of the $\mcQ_j^{\ast\ast}$ that
  \begin{align*}
    \chi_{\mcU_\lambda(\vecv)} \abs{\nabla \vecv_\lambda} \leq c\, \lambda.
  \end{align*}
  Since $\vecv_\lambda = \vecv$ on $\mcH_\lambda(\vecv)$, we get the
  first part of the claim
  \begin{align*}
    \abs{\nabla \vecv_\lambda} &\leq c\,
    \lambda \chi_{\mcU_\lambda(\vecv) \cap \Omega} + \abs{\nabla
      \vecv} \chi_{\mcH_\lambda(\vecv)}.
  \end{align*}
  Recall that $\mcH_\lambda(\vecv) = (\R^d \setminus \Omega) \cup
  \{M(\nabla \vecv)\leq \lambda\}$. Now, $\vecv_\lambda = 0$ on $\R^d
  \setminus \Omega$ and $\abs{\nabla \vecv} \leq M(\nabla \vecv)$
  proves that $\abs{\nabla \vecv} \chi_{\mcH_\lambda(\vecv)} \leq
  \lambda$. This proves the second part of the claim.
\end{proof}

The following theorem is the analogue of Corollary \ref{cor:dremlip} for
Sobolev functions. Similar results can be found
in~\cite{DieningMalekSteinhauer:08} and~\cite{BreDieFuc:12}.
\begin{cor}
  \label{cor:remlip}
  Let $1< s< \infty$ and let $\{\vec{e}^n\}_{n\in\N}\subset W^{1,s}_0(\Omega)^d$ be
  a sequence, which converges to zero weakly in
  $W^{1,s}_0(\Omega)^d$, as $n\to\infty$.

  Then, there exists a
  sequence $\{\lambda_{n,j}\}_{n,j\in\N}\subset\R$ with $2^{2^j} \leq
  \lambda_{n,j}\leq 2^{2^{j+1}-1}$ such that the Lipschitz truncations
  $\vec{e}^{n,j} := \vec{e}^n_{\lambda_{n,j}}$, $n,j\in\N$, have the following
  properties:
  \begin{enumerate}[leftmargin=1cm,itemsep=1ex,label={\rm (\alph{*})}]
  \item \label{itm:remlip1} $\vec{e}^{n,j}\in W^{1,\infty}_0(\Omega)^d$
    and $\vec{e}^{n,j}=\vec{e}^n$ on $\mcH_{\lambda_{n,j}}$;
  \item \label{itm:remlip2} $\|\nabla\vec{e}^{n,j}\|_\infty\leq
    c\lambda_{n,j}$;
  \item  \label{itm:remlip3} $\vec{e}^{n,j} \to 0$ in
    $L^\infty(\Omega)^d$ as $n \to
    \infty$;
  \item \label{itm:remlip4} $\nabla\vec{e}^{n,j} \weak^\ast 0$ in
    $L^\infty(\Omega)^{d\times d}$ as  $n
    \to \infty$;
  \item \label{itm:remlip5} For all $n,j \in \N$ we have
    $\norm[s]{\lambda_{n,j} \chi_{\mcU_\lambda(\vec{e}^n)}}
    \leq c\, 2^{-\frac{j}{s}}\norm[s]{\nabla \vece^n}$, with a
    constant $c>0$ depending on $s$.
  \end{enumerate}
\end{cor}

\begin{proof}
  The assertions follow by adopting the proof of Corollary \ref{cor:dremlip}.
\end{proof}


\end{document}